\documentclass[a4paper,12pt]{article}

\usepackage[cp1251]{inputenc}
\usepackage[english]{babel}
\usepackage{amsmath,amsthm,amsfonts,amssymb}
\allowdisplaybreaks
\usepackage{geometry}
\usepackage{enumerate}
\usepackage{url}

\usepackage[noadjust]{cite}

\theoremstyle{plain}
\newtheorem{theorem}{Theorem}
\newtheorem{lemma}{Lemma}
\newtheorem{corollary}{Corollary}
\theoremstyle{definition}

\theoremstyle{definition}
\newtheorem*{definition*}{Definition}
\theoremstyle{remark}
\newtheorem*{remark*}{Remark}

\newcommand{\N}{\mathbb{N}}

\newcommand{\abs}[1]{\left\lvert#1\right\rvert}
\newcommand{\set}[1]{\left\{#1\right\}}

\DeclareMathOperator{\diam}{diam}
\DeclareMathOperator{\Prob}{\mathsf{P}}
\DeclareMathOperator{\Expect}{\mathsf{E}}

\begin{document}

\title{On Bernoulli convolutions generated by  second Ostrogradsky series and their fine fractal properties}

\author{Sergio Albeverio$^{1,2,3,4}$  Mykola Pratsiovytyi$^{5,6}$, \\Iryna Pratsiovyta$^{7}$, Grygoriy Torbin$^{8,9}$  }
\date{}
\maketitle

\begin{abstract}
We study properties of  Bernoulli convolutions generated  by the second Ostrogradsky series, i.e.,  probability
distributions of random variables
\begin{equation}
\xi = \sum_{k=1}^\infty
\frac{(-1)^{k+1}\xi_k}{q_k},
\end{equation}
where $q_k$ is a sequence of positive integers with $q_{k+1}\geq q_k(q_k+1)$, and $\{\xi_k\}$ are independent
 random variables taking the values $0$ and $1$ with
probabilities $p_{0k}$ and $p_{1k}$ respectively. We prove that $\xi$ has
an anomalously fractal Cantor type singular distribution ($\dim_H (S_{\xi})=0$) whose Fourier-Stieltjes transform does not tend to zero at infinity.
We also develop different approaches how to estimate a level of "irregularity" of probability distributions whose spectra are of zero Hausdorff dimension. Using generalizations of the Hausdorff measures and dimensions,  fine fractal properties of the probability measure $\mu_\xi$ are studied in details. Conditions for the Hausdorff--Billingsley dimension preservation
on the spectrum by its probability distribution function are also
obtained.
\end{abstract}

%\maketitle

\medskip

$^1$~Institut f\"ur Angewandte Mathematik, Universit\"at Bonn,
Wegelerstr. 6, D-53115 Bonn (Germany); $^2$~SFB 611, Bonn, BiBoS,
Bielefeld--Bonn; $^3$~CERFIM, Locarno and Acc. Arch., USI
(Switzerland); $^4$~IZKS, Bonn; E-mail: albeverio@uni-bonn.de

$^5$~National Pedagogical University, Pyrogova str. 9, 01030 Kyiv
(Ukraine) $^{6}$~Institute for Mathematics of NASU,
Tereshchenkivs'ka str. 3, 01601 Kyiv (Ukraine); E-mail:
prats4@yandex.ru

$^8$~National Pedagogical University, Pyrogova str. 9, 01030 Kyiv
(Ukraine); E-mail: lightsoul2008@gmail.com

$^8$~National Pedagogical University, Pyrogova str. 9, 01030 Kyiv
(Ukraine) $^{9}$~Institute for Mathematics of NASU,
Tereshchenkivs'ka str. 3, 01601 Kyiv (Ukraine); E-mail:
torbin@wiener.iam.uni-bonn.de (corresponding author)

\medskip


\medskip


\section{Introduction}

About 1861  M.~V.~Ostrogradsky considered two algorithms for the
expansions of positive real numbers in alternating series:
\begin{equation}\label{eq:o1series}
\sum_k \frac{(-1)^{k+1}}{q_1q_2\dots q_k}, \quad \text{where
$q_k\in\N$, $q_{k+1}> q_k$},
\end{equation}
\begin{equation}\label{eq:o2series}
\sum_k \frac{(-1)^{k+1}}{q_k}, \quad \text{where $q_k\in\N$,
$q_{k+1} \geq q_k(q_k+1)$}
\end{equation}
(the first and the second Ostrogradsky series respectively). They
were found by E.~Ya.~Remez among manuscripts and unpublished
papers by M.~V.~Ostrogradsky~\cite{Rem51}. These series give  good rational approximations for real numbers.

One can prove (see, e.g., \cite{Pra98}) that for
any real number $x\in [0,1]$ there exists a sequence $\{q_k = q_{k}(x)\}$ such that $q_{k+1} \geq q_k(q_k+1)$ and
\begin{equation}\label{eq:repres.sec.Ostr.ser}
x =  \frac{1}{q_{1}(x)}-\frac{1}{q_{2}(x)}+\frac{1}{q_{3}(x)}+\ldots+\frac{1}{q_{k}(x)}+\ldots=\sum_{k=1}^{\infty}
\frac{(-1) ^{k+1}}{q_{k}(x)}.
\end{equation}
 If $x$ is irrational, then the expansion (\ref{eq:repres.sec.Ostr.ser}) is unique and it has an infinite number of terms. So, (\ref{eq:repres.sec.Ostr.ser}) establishes  one-to-one mapping between the set of
all infinite increasing sequences of positive integers
$(q_{1},q_{2}\ldots,q_{k},\ldots)$ with
$q_{k+1}\geq q_{k}(q_{k}+1)$ and the set of all irrational
numbers from the unit interval.
If $x$ is rational, then the expression
(\ref{eq:repres.sec.Ostr.ser}) has a finite number of terms and there exist exactly two different expansions of $x$ in form (\ref{eq:repres.sec.Ostr.ser}).

In the present paper we study properties of infinite
Bernoulli convolutions generated by series of the form
\eqref{eq:o2series}. Let us shortly recall that an infinite(general non-symmetric) Bernoulli convolutions with
  bounded spectra is the  distribution of random series
  $\sum\limits_{k=1}^{\infty} \xi_k b_k,$  where $~~ \sum\limits_{k=1}^{\infty}  |b_k|  < \infty$,
 and $\xi_k$ are independent random variables taking values $0$  and $1$ with probabilities $p_{0k}$ and $p_{1k}$ respectively.  Measures of this form
have been studied since 1930's from the pure probabilistic point of
view as well as for their applications in harmonic analysis, in the
theory of dynamical systems and in fractal analysis (see, e.g., \cite{ PeresSclagSolomyak98} for details and references).

The main purpose of the paper is to study properties of  Bernoulli convolutions generated  by the second Ostrogradsky series, i.e., the probability
distributions of random variables
\begin{equation}\label{eq:psi.def}
\xi = \sum_{k=1}^\infty
\frac{(-1)^{k+1}\xi_k}{q_k},
\end{equation}
where $q_k$ is an arbitrary second Ostrogradsky sequence, i.e., sequence of positive integers with $q_{k+1}\geq q_k(q_k+1)$, and $\{\xi_k\}$ is a sequence of independent random variables taking the values $0$ and $1$ with
probabilities $p_{0k}$ and $p_{1k}$ respectively
($p_{0k}+p_{1k}=1$). Since there exists a natural one-to-one correspondence between the set of irrational numbers of the unit interval and the set of infinite second Ostrogradsky sequences, we have a natural parametrization of  symmetric ($p_{0k}=\frac{1}{2}$) random variables of the form (\ref{eq:psi.def}) via the set of irrational numbers.

We prove that the distribution of $\xi$ is anomalously fractal, i.e., it is singular w.r.t. Lebesgue measure and its Hausdorff dimension is equal to zero.  We also study properties of the Fourier-Stieltjes transform and coefficients of these measures. In particular, we show  that this class of Bernoulli convolutions does not contain any Rajchman measure,  i.e., the Fourier-Stieltjes transform of the distribution of  $\xi$  does not tend to zero as $t$ tends to infinity. Moreover, in most cases the upper limit of the corresponding sequence of   the Fourier-Stieltjes coefficients is equal to 1.

Finite as well as infinite convolutions of distributions of random variables of the form (\ref{eq:psi.def}) are also studied in details. In particular, we show that any finite such a convolution has an anomalously fractal spectrum. For the limiting case   we prove that infinite convolution is of pure type (i.e., it is either purely discretely distributed or absolutely continuously resp. singularly continuously distributed), prove necessary and sufficient conditions for the discreteness and show that the Hausdorff dimension of such  an infinite convolution can vary from 0 to 1.

Let $\mu = \mu_{\xi}$ be the probability measure corresponding to $\xi$. Since the spectrum $S_{\mu}$ is of zero Hausdorff dimension and, therefore, the Hausdorff dimension of the measure $dim_H(\mu) := \inf\limits \{ \dim_H(E), E
\in \mathcal{B}, \mu(E)=1 \}$  is also equal to zero, we conclude that  the classical Hausdorff dimension does not
reflect the difference between the spectrum and other essential supports (see, e.g., \cite{TorbinUMJ2005}) of the singular measure $\mu$. Moreover, if $p^{'}_{0k}= p^{'} \in (0,1) $ and    $p^{''}_{0k}= p^{''} \in (0,1), p^{''} \neq p^{'}$, then the corresponding random variables $\xi^{'}$ and $\xi^{''}$ are mutually singularly distributed on the common spectrum $S_\mu$, and all of them are singularly continuous w.r.t. Lebesque measure.

Since the spectra of all random variables $\xi$  are ``very poor'' in both the metric and the fractal sense (their
classical Hausdorff dimensions are equal to zero), to  study fine fractal properties of the distribution of $\xi$ it is necessary to apply more delicate tools than $\alpha-$dimensional Hausdorff measure and the corresponding Hausdorff dimension. To this end we consider  the so called $h-$Hausdorff measures  (see, e.g.,\cite{Fal03})  and Hausdorff-Billingsley dimensions w.r.t. an appropriate probability measure (see, e.g., \cite{AT2} or Section \ref{sec:fractal.properties} for details). As an adequate example of such a measure we consider
the measure $\nu^*$ corresponding to the uniform distribution on the spectrum $S_\mu$. We find necessary and sufficient conditions for $\mu$ to be absolutely continuous resp. singular w.r.t. $\nu^*$,  prove a formula for the calculation of the
Hausdorff--Billingsley dimension of the spectrum  of
$\xi$ w.r.t. the measure $\nu^*$, and show that the distribution function of the measure  $\nu^*$ can be considered as a good choice for the gauge function $h(t)$. Moreover, in the same section  we study
internally fractal properties of $\xi$. In particular, we find the
Hausdorff--Billingsley dimension of the measure $\mu$ itself.

In the last section of the paper we develop third approach how to study a level of "irregularity" of probability distributions whose spectra are of zero Hausdorff dimension. We consider a problem of the preservation of the Hausdorff--Billingsley dimension of subsets of the spectrum under the distribution function $F_{\xi}$. For the case where  elements of the matrix $|p_{ik}|$ are bounded away from zero we find necessary and sufficient conditions for the dimension preservation.

% ----------------------------------------------------------------

% ----------------------------------------------------------------

\section{ Expansions of real numbers via the  second Ostrogradsky series.
 %as the apparatus  for the representation of the real numbers and some their properties
 }

\begin{definition*}
The numerical series of following form
\begin{equation}\label{eq:vugliad.sec.Ostr.ser}
\frac{1}{q_{1}}-\frac{1}{q_{2}}+\frac{1}{q_{3}}+\ldots+\frac{1}{q_{k}}+\ldots=\sum_{k=1}^{\infty}
\frac{(-1) ^{k+1}}{q_{k}},
\end{equation}
where $\{q_{k}\}$ is a sequence of positive integers with
\begin{equation}\label{eq:ymova}
q_{k+1}\geq q_{k}(q_{k}+1)
\end{equation}
is said to be the second Ostrogradsky series.
\end{definition*}
%\begin{theorem}
%As it was mentioned in the introduction
For any irrational number $x\in [0,1]$ there exists a unique sequence $\{q_k = q_{k}(x)\}$ such that
\begin{equation}\label{eq:repres.sec.Ostr.ser}
x =  \frac{1}{q_{1}(x)}-\frac{1}{q_{2}(x)}+\frac{1}{q_{3}(x)}+\ldots+\frac{1}{q_{k}(x)}+\ldots=\sum_{k=1}^{\infty}
\frac{(-1) ^{k+1}}{q_{k}(x)}.
\end{equation}
%%%
%%% If $x$ is irrational, then the expansion (\ref{eq:repres.sec.Ostr.ser}) is unique and it has an infinite number of terms. So, (\ref{eq:repres.sec.Ostr.ser}) establishes  one-to-one mapping between the set of
%%%all infinite increasing sequences of positive integers
%%%$(q_{1},q_{2}\ldots,q_{k},\ldots)$ with
%%%$q_{k+1}\geq q_{k}(q_{k}+1)$ and the set of all irrational
%%%numbers from the unit interval.
%%%If $x$ is rational, then the expression
%%%(\ref{eq:repres.sec.Ostr.ser}) has a finite number of terms and there exist exactly two different expansions of $x$ in form (\ref{eq:repres.sec.Ostr.ser}).

The sequence $\{q_{k}\}$ can be determined via the
following algorithm:
\begin{equation} \label{ar:algorithm}
\left\{
 \begin{array}{lll}
1=q_{1}x+\beta_{1}~~~(0\leq\beta_{1}<x),
\\q_{1}=q_{2}\beta_{1}+\beta_{2}~~~(0\leq\beta_{2}<\beta_{1}),
\\q_{2}q_{1}=q_{3}\beta_{2}+\beta_{3}~~~(0\leq\beta_{3}<\beta_{2}),
\\\ldots\ldots\ldots\ldots\ldots\ldots\ldots\ldots
\\q_{k}\ldots q_{2}q_{1}=q_{k+1}\beta_{k}+\beta_{k+1}~~~(0\leq\beta_{k+1}<\beta_{k}),
\\\ldots\ldots\ldots\ldots\ldots\ldots\ldots\ldots
\end{array}
\right.
\end{equation}

Let us consider several  examples of the second Ostrogradsky series:

$$1)~
\frac{1}{1} -\frac{1}{1\cdot2}+\frac{1}{2\cdot3}-\frac{1}{6\cdot7}+\frac{1}{42\cdot43}-\frac{1}{1806\cdot1807}+\ldots.$$ Here $q_1 = 1,$ and $q_{k+1} = q_{k}(q_{k}+1), \forall k \in N.$

$$2)~\frac{1}{s}-\frac{1}{s^{3}}+\frac{1}{s^{7}}-\frac{1}{s^{15}}+\frac{1}{s^{31}}-\ldots=\sum_{k=1}
^{\infty}\frac{(-1)^{k-1}}{s^{m_{k}}},$$ where $2\leq s \in N, $
$ m_{1}=1, $ $ m_{k}=2m_{k-1}+1;$

$$3)~
\frac{1}{2\cdot3}-\frac{1}{2^{4}\cdot 3}+\frac{1}{2^{8}\cdot
3^{8}}-\frac{1}{2^{8}\cdot 3^{32}}+\frac{1}{2^{64}\cdot
3^{64}}-\frac{1}{2^{256}\cdot 3^{64}}+\ldots;
$$

$$4)~
\frac{1}{2}-\frac{1}{7}+\frac{1}{59}-\frac{1}{3541}+\ldots=\sum_{k=1}^{\infty}
\frac{(-1) ^{k-1}}{p_{k}},$$ where $\{p_{k}\}$ is an infinite
sequence of prime numbers such that  $$p_{1}=2, ~~ p_{2}=7, ~~
p_{3}=59, ~ \ldots,$$ $p_{k}$ is the minimal prime number such that
$p_{k} \ge p_{k-1}(p_{k-1}+1).$

Let us mention some evident properties of denominators of the second Ostrogradsky series.

\begin{enumerate}
\item[1.] $ \dfrac{q_{n}}{q_{n+1}}\leq \frac{1}{q_{n}+1}< \frac{1}{q_{n}};$

\item[2.] $ \lim\limits_{n \rightarrow \infty}\dfrac{q_{n}}{q_{n+1}}=0;$

\item[3.] $ q_{n}\geq n!;$

%%%\item[4.] $ \left\{
%%%\begin{array}{lll}
%%%q_{2}\geq q_{1}(q_{1}+1)=q_{1}^{2}+q_{1}\geq q_{1}^{2}+1\geq
%%%q_{1}^{2};\\ q_{3}\geq q_{2}(q_{2}+1)=q_{2}^{2}+q_{2}\geq
%%%q_{2}q_{1}(q_{1}+1)+q_{2};\\
%%%\ldots\ldots\ldots\ldots\ldots\ldots\ldots\ldots\ldots \\
%%%q_{n+1}\geq q_{n}(q_{n}+1)\geq q_{n}q_{n-1}\ldots
%%%q_{2}q_{1}(q_{1}+1)+q_{n}\ldots q_{2}+\ldots +q_{n}q_{n-1}+q_{n};
%%%\end{array}
%%%\right.$

\item[4.] $ \left\{
\begin{array}{lll}
q_{2}\geq q_{1}(q_{1}+1)> q_{1}^{2};\\
q_{3}  > q^2_{2}  > q^4_{1};\\
\ldots\ldots\ldots\ldots\ldots\ldots\ldots\ldots\ldots \\
q_{n+1}> q^2_{n} > q^4_{n-1} > q^8_{n-2}>...>q^{2^{k+1}}_{n-k}> ...>q^{2^{n-1}}_{2}>q^{2^{n}}_{1} ;
\end{array}
\right.$

\item[5.]
%$q_{n+2}>q_{2}^{2^{n}}>q_{1}^{2^{n+1}};$ \\
 $q_{n+1}>{q_{2}}^{{2}^{n-1}}\geq 2^{{2}^{n-1}}. $

%%%\item[6.]   For the any natural  $n$ the following inequality
%%%$$\frac{q_{n}}{q_{n+1}}\leq \frac{1}{q_{n}+1}\leq
%%%\frac{1}{q_{n}}-\frac{1}{q_{n+1}}<\frac{1}{q_{n}}.$$
%%%\begin{proof}
%%%$$\frac{1}{q_{n}}-\frac{1}{q_{n+1}}=\frac{q_{n+1}-q_{n}}{q_{n}q_{n+1}}\geq
%%%\frac{{q_{n}}^{2}+q_{n}-q_{n}}{q_{n}q_{n+1}}=\frac{q_{n}}{q_{n+1}},$$
%%%$$\frac{1}{q_{n}+1}-\frac{q_{n}}{q_{n+1}}=\frac{q_{n+1}-{q_{n}}^{2}-q_{n}}{(q_{n}+1)q_{n+1}}\geq
%%%0.$$
%%%\end{proof}

\item[6.] $ \sum\limits_{i=n+1}^{\infty}\dfrac{1}{q_{i}}<\dfrac{2}{q_{n+1}},$
%\begin{proof}

because
$\sum\limits_{i=n+1}^{\infty}\frac{1}{q_{i}}\leq
\frac{1}{q_{n+1}}+\sum\limits_{i=1}^{\infty}\frac{1}{q_{n+i}(q_{n+i}+1)}<
\frac{1}{q_{n+1}}+\frac{1}{q_{n+1}}\sum\limits_{i=1}^{\infty}\frac{1}{q_{n+i}+1}\leq$

$\leq\frac{1}{q_{n+1}} \left(
1+\sum\limits_{i=1}^{\infty}\frac{1}{q_{n+i}} \right)< \frac{1}{q_{n+1}}
\left(1+\frac{1}{2}+\frac{1}{2^2}+\frac{1}{2^3} + \ldots \right) = \frac{2}{q_{n+1}}.$
%\end{proof}

\item[7.]
$ \dfrac{q_{1}\ldots
q_{n}}{q_{n+1}}\leq \dfrac{q_1}{q_{2}+1+\dfrac{1}{q_{2}}+\dfrac{1}{q_{2}q_{3}}+ \ldots
+\dfrac{1}{q_{2}\ldots q_{n-1}}} <
\dfrac{q_1}{q_{2}+1+\dfrac{1}{q_{2}}}< \frac{2}{7};
$

$ \dfrac{q_{1}\ldots
q_{n}}{q_{n+1}}\leq
\dfrac{1}{q_{1}+1+\dfrac{1}{q_{1}}+\dfrac{1}{q_{1}q_{2}}+ \ldots
+\dfrac{1}{q_{1}\ldots q_{n-1}}}<\dfrac{1}{q_{1}+1}.$

\item[8.] If $a_n :=\frac{1}{q_n}$ and $r_n:= \sum\limits_{i=n+1}^{\infty}\dfrac{1}{q_{i}},$ then $a_n >r_n, \forall n \in N, $
%\begin{proof}

because
$\frac{a_n}{r_n} > \frac{ \frac{1}{q_n}}{\frac{2}{q_{n+1}}} = \frac{q_{n+1}}{2 q_n} \geq \frac{q_n(q_n+1)}{2 q_n} = \frac{q_n+1}{2} \geq 1.$
%\end{proof}

\end{enumerate}
%%%\begin{lemma}\label{lm:vlast.chleniv}
%%%The terms of second Ostrogradsky series have following property:
%%%\begin{equation}
%%%\label{eq:vlast.ak}
%%%a_{k}>r_{k}~~~\forall~k\in N.
%%%\end{equation}
%%%\end{lemma}
%%%\begin{proof}
%%%$$a_{k}=\frac{1}{q_{k}},$$
%%%$$r_{k}=\frac{1}{q_{k+1}}+\frac{1}{q_{k+2}}+ \ldots <
%%%\frac{1}{q_{k}^{2}}+\frac{1}{q_{k}^{4}}+\frac{1}{q_{k}^{8}}\leq$$
%%%$$\leq\frac{1}{q_{k}^{2}}+\frac{1}{q_{k}^{4}}+\frac{1}{q_{k}^{6}}+\frac{1}{q_{k}^{8}}+\ldots$$
%%%We have geometric progression with $b_{1}=\frac{1}{{q_{k}}^{2}}$
%%%and $q=\frac{1}{{q_{k}}^{2}}$, so $$r_{k}\leq
%%%\frac{b_{1}}{1-q}=\frac{\frac{1}{q_{k}^{2}}}{1-\frac{1}{q_{k}^{2}}}=\frac{1}{q_{k}^{2}-1}$$
%%%$$=\frac{1}{(q_{k}+1)(q_{k}-1})\leq \frac{1}{q_{k}+1}.$$ Hence
%%%$$a_{k}=\frac{1}{q_{k}}>\frac{1}{q_{k}+1}>r_{k}.$$
%%%\end{proof}

\section{Probability distributions generated by the second Ostrogradsky series and their convolutions}
\label{sec:random.incomplete.sum}

Let $r$ be an irrational number from the unit interval and let $\{q_k = q_k(r)\}$ be the the second Ostrogradsky  sequence corresponding to the number $r$, i.e., $\{q_k\}$ is a unique infinite sequence of  positive integers satisfying condition (\ref{eq:ymova}), and such that
$$
r =  \frac{1}{q_{1}(r)}-\frac{1}{q_{2}(r)}+\frac{1}{q_{3}(r)}+\ldots+\frac{1}{q_{k}(r)}+\ldots=\sum_{k=1}^{\infty}
\frac{(-1) ^{k-1}}{q_{k}(r)}.
$$
 Let
\[
L=\{e: e= (e_1,e_2,\dots,e_k,\dots), e_k\in \{0,1\} \}.
\]
The sum $s=s(\{e_k\})$ of the series
\begin{equation}\label{eq:incomplete.sum}
\sum_{k=1}^\infty \frac{(-1)^{k-1}e_k}{q_k}, \quad \text{where
$\{e_k\}\in L$},
\end{equation}
is said to be an \emph{incomplete sum of the
series~(\ref{eq:vugliad.sec.Ostr.ser})}. It is clear that $s$
depends on the whole {\it infinite} sequence $\{e_k\}$.
 We denote the
expression~\eqref{eq:incomplete.sum} and its sum $s$ formally by
$\Delta_{e_1 e_2 \dots e_k {\bf \dots}} ~$.
%%% (the reason for such a
%%%notation will be clarified later in Lemma \ref{lem:properties},
%%%prop.5).
Any partial sum of the
series~(\ref{eq:vugliad.sec.Ostr.ser}) is its incomplete sum.
In a very special case where $e_k=1, ~ \forall k \in N$, we obtain the "complete" sum.
% It
%is evident that the set of incomplete sums of the
%series~(\ref{eq:vugliad.sec.Ostr.ser}) is of continuum
%cardinality.

Let  $C_r$ be the set of all incomplete sums of the series
(\ref{eq:vugliad.sec.Ostr.ser}). For any $s\in C_r$ there
exists a sequence $\{e_k\}=\{e_k(s)\}\in L$ such that
\[
s=\sum_{k=1}^\infty \frac{(-1)^{k-1}e_k(s)}{q_k}.
\]

Let $c_1$, $c_2$,~\dots, $c_m$ be a fixed sequence consisting of
zeroes and ones. The set $\Delta'_{c_1c_2\dots c_m}$ of all
incomplete sums $\Delta_{c_1c_2\dots c_m e_{m+1}\dots
e_{m+k}\dots}$, where $e_{m+j}\in \{0,1 \}$ for any $j \in N$, is
called the \emph{cylinder} of rank $m$ with base $c_1 c_2\dots
c_m$. It is evident that
\[
\Delta'_{c_1c_2\dots c_m a_{m+1}}\subset\Delta'_{c_1c_2\dots c_m}, \quad
\forall a_{m+1}\in\{0,1\}.
\]

The closed interval
\[
\left[s_m-\sum_{i:2i>m} \frac{1}{q_{2i}}, s_m+\sum_{i:2i-1>m}
\frac{1}{q_{2i-1}}\right] \mbox{with } s_m:=\sum\limits_{k=1}^m \frac{(-1)^{k-1}c_k}{q_k}
\]
is  said to be the \emph{cylindrical interval of rank $m$ with base $c_1c_2\dots
c_m$} ($c_i\in \{0,1\}$).
We denote it symbolically by $\Delta_{c_1c_2\dots c_m}$.

\begin{lemma}\label{lem:properties}
The cylindrical intervals have the following properties:
\begin{enumerate}[\upshape1.]
\item\label{enu:properties.1}
$\Delta'_{c_1c_2\dots c_m}\subset \Delta_{c_1c_2\dots c_m}$.

\item\label{enu:properties.2} $\Delta_{c_1c_2\dots
c_m}=\Delta_{c_1c_2\dots c_m0}\bigcup \Delta_{c_1c_2\dots c_m1}$.

\item\label{enu:properties.3} The length of $\Delta_{c_1c_2\dots
c_m}$ is equal to
\begin{align*}
\abs{\Delta_{c_1c_2\dots c_m}}&=\diam\Delta'_{c_1c_2\dots c_m}=
 \sum\limits_{k=m+1}^\infty
\frac1{q_k}<\frac{2}{q_{m+1}}\to0\quad (m\to\infty).
\end{align*}

\item\label{enu:properties.4} $\Delta_{c_1c_2\dots
c_m0}\cap\Delta_{c_1c_2\dots c_m1}=\varnothing$.

\item\label{enu:properties.5} $\bigcap\limits_{m=1}^\infty
\Delta_{c_1c_2\dots c_m}= \bigcap\limits_{m=1}^\infty
\Delta'_{c_1c_2\dots c_m} = s =: \Delta_{c_1c_2\dots
c_m\dots}$ for any $\{c_k\}\in L$.

\item\label{enu:properties.6} $C_r=\bigcap\limits_{m=1}^\infty
\bigcup\limits_{c_i\in \{0,1\}} \Delta_{c_1c_2\dots
c_m}$.
\end{enumerate}
\end{lemma}

As it has been mentioned above, for the any irrational  $r\in [0,1]~~$ there exists the
sequence $\{q_{k}\}:~x=\sum\limits_{k=1}^{\infty} \frac{(-1)
^{k+1}}{q_{k}}.$  The sequence $\{q_{k}\}$ generates the random
variable of the following form:
\begin{equation}\label{eq:rand.var.def}
\xi=\sum_{k=1}^\infty\frac{(-1)^{k+1}\xi_k}{q_{k}}.
\end{equation}

 This random variable can be represented as a "shifted" infinite Bernoulli convolution (see, e.g., \cite{AT3}), generated by positive convergent series $\sum\limits_{k=1}^\infty a_k$ with $a_k = \frac{1}{q_{k}}$:
\begin{equation}\label{eq:shifted Bernoulli convolution}
\xi=\sum_{k=1}^\infty\frac{(-1)^{k+1}\xi_k}{q_{k}}=
\sum\limits_{k=1}^{\infty}\frac{\eta_{k}}{q_{k}}-2\sum\limits_{k=1}^{\infty}\frac{1}{q_{2k}}=
\sum\limits_{k=1}^{\infty}\frac{\eta_{k}}{q_{k}}- const = \eta - const.
\end{equation}
where $\eta_k$ also takes values $0$ and $1$.  So, random variables $\xi$ and $\eta$ have equivalent distributions.

According to the Jessen--Wintner theorem~(see, e.g.,\cite{JW35}) the random
variable $\xi$ is of pure type, i.e., it is either pure discrete
or pure absolutely continuous resp. pure singularly continuous w.r.t.  Lebesgue measure.

The following proposition follows directly from the P.~L\'evy
theorem~ (\cite{Lev31}) and it gives us necessary and sufficient
conditions for the the continuity of $\xi.$

\begin{lemma}\label{cor:xi.continuous}
The random variable $\xi$ has a continuous  distribution if and only
if
\[
D =\prod_{k=1}^\infty \max\{p_{0k},p_{1k}\}=0.
\]
\end{lemma}
%%%
%%%\begin{corollary}\label{cor:xi.continuous}
%%%The random variable $\xi$ has a continuous distribution if and
%%%only if $P_{\max}=0$.
%%%\end{corollary}

In the sequel we shall be interested in continuous distributions
only.

The \emph{spectrum \textup(topological support\textup)} $S_\xi$ of
the distribution of the random variable $\xi$ is the minimal closed
support of $\xi$. It is clear that  $S\mu$ is a perfect set (i.e., a closed set without isolated points).
Since \begin{align*}
S_\xi=\set{x: \Prob\{\xi\in(x-\varepsilon, x+\varepsilon)\}>0\;
\forall\,\varepsilon>0}=
 \set{x: F_\xi(x+\varepsilon)-F_\xi(x-\varepsilon)>0\;
\forall\,\varepsilon>0},
\end{align*}
where $F_{\xi}$ is the distribution function of the random
variable $\xi$, we deduce that $$S_\xi=\{x: x \mbox{~can be represented in the form~} \sum\limits_{k=1}^{\infty}\frac{(-1)^{k+1}e_{k}}{ q_{k}},~e_{k}\in~\{0,1\} \mbox{with} p_{e_{k} k}>0\}.$$

%The minimal closed set supporting the distribution of a random variable $\varphi$ is said to be a spectrum (topological support)
%of $\varphi.$

\begin{theorem}
 The spectrum $S_\xi$ of the distribution of the random variable
 $\xi$ is a nowhere dense  set of zero Hausdorff dimension.
\end{theorem}
\begin{proof}
From (\ref{eq:shifted Bernoulli convolution}) it follows that topological and fractal properties of random variables $\xi$ and $\eta$ are the same. So, we can apply general results on Bernoulli convolutions (see, e.g., \cite{AT3}).
Since the sequence $\{q_k\}$ is strictly decreasing and $a_k=\frac{1}{q_k}>r_k=\sum\limits_{i=k+1}^{\infty}a_k=\sum\limits_{i=k+1}^{\infty}\frac{1}{q_i}$, we conclude that $S_\xi$ is a nowhere dense set and its Hausdorff dimension can be calculated by the following formula:
$$
\dim_H(S_{\xi}) = \liminf\limits_{k \to \infty} \frac{k \ln 2}{-
\ln r_k}.
$$

From properties (5) and (6) of denominators of the second Ostrogradsky series it follows that $r_{k}=\frac{1}{q_{k+1}}+\frac{1}{q_{k+2}}+\frac{1}{q_{k+3}}+\ldots < \frac{2}{q_{k+1}}< \frac{2}{2^{2^{k-1}}}$.
So, $$
\dim_H(S_{\xi}) = \liminf\limits_{k \to \infty} \frac{k \ln 2}{-
\ln r_k} \leq \lim\limits_{k \to \infty} \frac{k \ln 2}{-
\ln \frac{2}{2^{2^{k-1}}}} = 0,$$
which proves the theorem.
\end{proof}

\begin{corollary}
If $D =0$, then $\xi$ has a singularly continuous distribution of the
Cantor type with an anomalously fractal spectrum.
\end{corollary}
%%%
%%%\begin{proof}
%%%If $P_{\max}=0$, then $\xi$ has a continuous distribution, due to
%%%Corollary~\ref{cor:xi.continuous} after
%%%Theorem~\ref{thm:xi.discrete}. Since the Hausdorff--Besicovitch
%%%dimension $\alpha_0(C_r)=0$ and $S_\xi\subset C_r$, the
%%%topological support $S_\xi$ is an anomalously fractal set. Since
%%%the Lebesgue measure $\lambda(C_r)=0$, we have $\lambda(S_\xi)=0$.
%%%So, $\xi$ has a distribution of the Cantor type.
%%%\end{proof}

% ----------------------------------------------------------------

\section{Properties of  Fourier--Stieltjes transorm of probability distributions generated by the second Ostrogradsky series} \label{sec:characteristic.function}

Let us consider the \emph{characteristic function $f_{\xi}(t)$ of the
random variable $\xi$} (Fourier--Stieltjes transorm of the corresponding probability measure), i.e.,
\[
f_\xi(t)=\Expect\left(e^{it\xi}\right).
\]
%%%%The theory of characteristic functions is convenient for the
%%%%investigation of the structure and properties of distributions of
%%%%real-valued random variables.

It is well known that for a singularly  continuous distribution with the distribution function $F(x)$ one has:
$$\sum\limits_{m=1}^{\infty}{c_{m}}^{2}=\infty,$$
where $c_{m}$ are the Fourier-Stieltjes coefficients of $F(x)$, i.e., $$ c_{m}=\int_{-\infty}^{\infty}e^{2\pi m i x} dF(x) = f_{\xi}(2\pi m).$$
Nevertheless for some classes of singular measures $c_{m}$ can tend to $0$ like for  absolutely continuous distributions.
\begin{lemma}
The characteristic function of random variable $\xi$ defined by (\ref{eq:rand.var.def}) is of the following form
\begin{equation}\label{eq:character.func}
f_\xi(t)=\prod_{k=1}^\infty f_k(t) \quad \text{with} \quad
f_k(t)=p_{0k}+p_{1k}\exp\frac{(-1)^{k-1}it}{q_k},
\end{equation}
and its absolute value is equal to
\[
\abs{f_\xi(t)} = \prod_{k=1}^\infty \abs{f_k(t)}, \quad
\text{where} \quad \abs{f_k(t)} =
\sqrt{1-4p_{0k}p_{1k}\sin^2\frac{t}{2q_k}}.
\]
\end{lemma}

\begin{proof}
Using the definition of a characteristic function and properties
of expectations, we have
\begin{align*}
f_\xi(t) &= \Expect\left(e^{it\xi}\right)=\Expect\left(\exp
\left(it\sum_{k=1}^\infty
\frac{(-1)^{k-1}\xi_k}{q_k}\right)\right)=\\ &=
\Expect\left(\exp\frac{it\xi_1}{q_1}\cdot
\exp\frac{-it\xi_2}{q_2}\cdot\dots\cdot
\exp\frac{(-1)^{k-1}it\xi_k}{q_k}\cdot\dotsm\right)=\\ &=
\prod_{k=1}^\infty \Expect\left(\exp
\frac{(-1)^{k-1}it\xi_k}{q_k}\right)= \prod_{k=1}^\infty
\left(p_{0k}+p_{1k}\exp\frac{(-1)^{k-1}it}{q_k}\right)=\\ &=
\prod_{k=1}^\infty
\left((p_{0k}+p_{1k}\cos\frac{(-1)^{k-1}t}{q_k}) +
i(p_{1k}\sin\frac{(-1)^{k-1}t}{q_k})\right)=\prod_{k=1}^\infty
f_k(t),
\end{align*}
and
\[
\abs{f_k(t)} =
\sqrt{p_{0k}^2+2p_{0k}p_{1k}\cos\frac{t}{q_k}+p_{1k}^2}=
\sqrt{1-4p_{0k}p_{1k}\sin^2\frac{t}{2q_k}},
\]
which proves the lemma.
\end{proof}
\begin{corollary}
The Fourier-Stieltjes coefficients of the distribution function of the random variable $\xi$ defined by equality
(\ref{eq:rand.var.def}) are of the following form $$c_{m}=\prod\limits_{k=1}^{\infty}\left(p_{ok}~+p_{1k}\exp\frac{(-1)^{k-1} 2\pi mi}{q_{k}}\right).$$
\end{corollary}
%%\begin{proof}
%%From lemma 5 for the characteristic function $f_{\xi}(t)$ of the random variable $\xi$ and from $c_{m}=f_{\xi}(2\pi m)$
%%statement of this corollary follows.
%%\end{proof}

Let $l.c.m.(m_1,m_2,...,m_k):= \min\limits_{M \in N} \{M: M\vdots m_1, M\vdots m_2, ..., M\vdots m_k\}.$

\begin{theorem}\label{thm:ck}
For any sequence $\{q_n\}$ generated by the second Ostrogradsky series,  the Fourier-Stieltjes coefficients of the distribution function of the random variable
$\xi$ have the following properties:
\begin{enumerate}[\upshape1)]
\item\label{ck:properties.1}
$\limsup\limits_{k\rightarrow\infty} |c_{k}|  >0;$
\end{enumerate}
\begin{enumerate}[\upshape2)]\label{liminf}
\item\label{ck:properties.2}
if~ \begin{equation}\liminf\limits_{k\rightarrow \infty}
\frac{l.c.m.(q_{1},q_{2},\ldots q_{n})}{q_{n+1}}=0,
\end{equation} ~ then~
$\limsup\limits_{k\rightarrow\infty} |c_{k}|=1;$

\end{enumerate}
%%%\begin{enumerate}[\upshape3)]
%%%\item\label{ck:properties.3}
%%%if~ $\liminf\limits_{k\rightarrow \infty}
%%%\frac{l.c.m.(q_{1},q_{2},\ldots q_{n})}{q_{n+1}}>0,$ ~ then~
%%%$\limsup\limits_{k\rightarrow\infty}|c_{k}|<1.$
%%%\end{enumerate}
\end{theorem}

\begin{proof}
From the definition follows that $c_{k}=f_{\xi}(2\pi k).$

1.If  $k_n=l.c.m.(q_1,q_2,\dots,q_n)$, then $c_{k_{n}}=f_{\xi}(l.c.m.(q_1,q_2,\dots,q_n)2\pi)$.
Let us estimate
$$|c_{k_{n}}|=|f_{\xi}(2\pi k_{n})|=\prod_{k=1}^\infty
\sqrt{1-4p_{0k}p_{1k}\sin^2\frac{ \pi k_{n}}{q_k}}
\geq \prod_{k=1}^\infty \sqrt{1-\sin^2\frac{\pi k_n}{q_k}} =$$ $$=\prod_{k=1}^\infty
\sqrt{1-\sin^2\frac{l.c.m.(q_{1},q_{2},\ldots,
q_{n})\pi}{q_k}}=\sqrt{\prod_{k=1}^\infty
{1-\sin^2\frac{l.c.m.(q_{1},q_{2},\ldots,
q_{n})\pi}{q_k}}}\geq$$ $$\geq\prod_{k=1}^\infty
\left({1-\sin^2\frac{l.c.m.(q_{1},\ldots,
q_{n})\pi}{q_k}}\right)=1\cdot \ldots \cdot 1\cdot\prod_{k=n+1}^\infty
\left({1-\sin^2\frac{l.c.m.(q_{1}, \ldots ,
q_{n})\pi}{q_k}}\right)\geq$$ $$\geq\prod_{k=n+1}^\infty
\left({1-\sin\frac{l.c.m.(q_{1},q_{2},\ldots,
q_{n})\pi}{q_k}}\right)\geq\prod_{k=n+1}^\infty
\left({1-\frac{l.c.m.(q_{1},q_{2},\ldots,
q_{n})\pi}{q_k}}\right)=A.$$
Since $$\frac{q_1}{q_{2}+1+\frac{1}{q_{2}}}\leq \frac{2}{7},~~\forall~q_{1}\in N; ~q_{2}\in N, \mbox{with} ~q_{2}\geq q_1(q_1+1),$$ we have $$\frac{l.c.m.(q_{1},q_{2},\ldots,q_{n})\pi}{q_{n+1}}\leq
\frac{q_{1}\cdot q_{2}\cdot\ldots\cdot q_{n}\cdot\pi}{q_{n+1}}\leq\frac{q_1 \pi}{q_{2}+1+\frac{1}{q_{2}}}\leq\frac{2 \pi}{7}.$$

It is well known that  $$(1-b_{1})(1-b_{2}) \cdot \ldots \cdot (1-b_{k}) \cdot ... \geq 1-(b_{1}+\ldots+b_{k}+ ...)$$ for any sequence $\{b_k\}$ with $0<b_{k}<1$. So,
%%%So, $1-\frac{l.c.m.(q_{1},q_{2},\ldots,
%%%q_{n})\pi}{q_{n+1}} \geq 1 - \frac{2 \pi}{7}. $
 $$ \prod\limits_{k=n+2}^{\infty}\left(1-\frac{l.c.m.(q_{1},q_{2},\ldots,q_{n})\pi}{q_{k}}\right)\geq 1-\sum_{k=n+2}^{\infty}\left(\frac{l.c.m.(q_{1},q_{2},\ldots,q_{n})\pi}{q_{k}}\right)=$$
$$=1-\sum\limits_{k=n+2}^{\infty}\frac{l.c.m.(q_{1},q_{2},\ldots,q_{n})\pi}{q_{n+1}}\cdot\frac{q_{n+1}}{q_{k}}\geq 1-\frac{\pi}{3}\cdot
\sum\limits_{k=n+2}^{\infty}\frac{q_{n+1}}{q_{k}}=B.$$
Since $$\frac{q_{n+1}}{q_{n+2}}\leq \frac{1}{q_{n+1}}<\frac{1}{2^{2^{n-1}}};$$
$$\frac{q_{n+1}}{q_{n+m}}=\frac{q_{n+1}}{q_{n+2}}\cdot\frac{q_{n+2}}{q_{n+3}}\cdot\ldots\cdot\frac{q_{n+m-1}}{q_{n+m}}< \frac{1}{2^{2^{n-1}}}\cdot\frac{1}{2^{2^{n}}}\cdot\ldots\cdot\frac{1}{2^{2^{n+m-3}}}\leq\frac{1}{2^{2^{n-1}}}\cdot\frac{1}{2^{m-2}},$$
we get $$B\geq 1-\frac{\pi}{3}\cdot\frac{1}{2^{2^{n-1}}}\left(1+\frac{1}{2}+\frac{1}{4}+\ldots\right)=1-\frac{2\pi}{3\cdot2^{2^{n-1}}}.$$
Therefore, $$A\geq\left(1-\frac{2\pi}{7}\right)\left(1-\frac{2\pi}{3\cdot2^{2^{n-1}}}\right)
>\left(1-\frac{2\pi}{7}\right)\left(1-\frac{\pi}{6}\right),~~\forall~n\in N.$$
Hence, $$\limsup\limits_{k\rightarrow\infty}|c_{k}|  \geq \limsup_{n\rightarrow\infty}|c_{k_{n}}|\geq \left(1-\frac{2\pi}{7}\right)\left(1-\frac{\pi}{6}\right)>0~.$$

2. If condition (\ref{liminf}) holds, then
%$$\liminf_{n\rightarrow \infty}\frac{l.c.m.(q_{1},q_{2},\ldots q_{n})}{q_{n+1}}=0 $$ is
there exists a sequence $\{n_{s}\}$ of positive integers such that $$\liminf_{n\rightarrow \infty} \frac{l.c.m.(q_{1},q_{2},\ldots q_{n_{s}})}{q_{n_{s}+1}}=0.$$
Let us
consider the sequence $k_{n_{s}}= l.c.m.(q_{1},q_{2},\ldots,
q_{n_{s}})$.
Using our previous arguments, we have
$$|c_{k_{n_{s}}}|\geq\left(1-\frac{l.c.m.(q_{1},q_{2},\ldots,q_{n_{s}})}{q_{n_{s}+1}}\right)\cdot\left(1-\frac{\pi}{3}\cdot
\frac{2}{2^{2^{n_{s}-1}}}\right)\rightarrow 1~~(s\rightarrow\infty).$$
Therefore,
$$\limsup\limits_{k\rightarrow\infty}|c_{k}| =1.$$\end{proof}

 For a given random variable
$\zeta$ one can define
\[
L_\xi=\limsup_{\abs{t}\to\infty}\abs{f_\xi(t)}.
\]
It is well known~
%\cite[p.~28]{Luk70}
that $L_\zeta=1$ for any discretely distributed random variable $\zeta$. If $\zeta$ has an absolutely continuous distribution, then $L_\zeta=0$.
 If $\zeta$ has a singularly continuous distribution, then $L_\zeta \in [0,1].$ More precisely: for
any real number $\beta \in[0,1]$ there exists a singularly distributed random
variable $\zeta_\beta$ such that $L_{\zeta_\beta}=\beta$.
 Let us stress the asymptotic behaviour at infinity of the absolute value of the
characteristic function of the random variable $\xi$,defined by (\ref{eq:rand.var.def}).
\begin{corollary}
For any sequence $\{q_n\}$ generated by the second Ostrogradsky series, we have
 $$L_{\xi}>0.$$
 If $~~~~~ \liminf\limits_{k \to \infty}\frac{l.c.m.(q_{1},q_{2},\ldots q_{n})}{q_{n+1}}=0,~~~~$ then $L_{\xi}=1.$
\end{corollary}
\begin{proof}
 It is clear that  $$L_{\xi}=\limsup\limits_{|t|\rightarrow\infty}|f_{\xi}(t)|\geq \limsup\limits_{k\rightarrow\infty}|f_{\xi}(2\pi k)|=\limsup\limits_{k\rightarrow\infty}|c_{k}|,$$
and the statement of corollary follows directly  from statements 1) and 2) of the latter theorem.
\end{proof}

\section{Convolutions of singular distributions generated by the second Ostrogradsky series}
It is well known (see, e.g., \cite{Luk70}) that for a random variable $\zeta$ which is a sum of two independent random variables $\zeta_1$ and $\zeta_2$ one has $F_{\zeta} = F_{\zeta_1}*F_{\zeta_2}  $ and  $f_{\zeta} = f_{\zeta_1} \cdot f_{\zeta_2} , $ where $F_{\zeta}$ resp. $f_{\zeta}$ means the probability distribution function resp. characteristic function of the random variable $\zeta$. If either $\zeta_1$ or $\zeta_2$ has an absolutely continuous distribution then $\zeta$ also has a density. If both $\zeta_1$ and  $\zeta_2$ has singular distribution then, generally speaking, nothing known about the distribution of  $\zeta$. The convolution of two singular probability distributions is either
singular or absolutely continuous, or is of a mixed type. An absolutely continuous distribution can arise even as a convolution of two anomalously fractal singularly continuous distributions.
%  It is easy to construct corresponding examples.
 The desirability of finding a criterium for the singularity resp. absolute continuity of the
convolution of two singular distributions has been expressed by many authors (see, e.g., \cite{Luk70}), but it is still unknown. It can however be given for special classes of random variables (see, e.g., \cite{AGPT} and references therein).

In this Section we study finite as well as infinite convolutions of probability distributions generated by the second Ostrogradsky series.

\subsection{Autoconvolutions}
\begin{theorem}
 Let $$\xi^{(j)}=\sum_{k=1}^{\infty}\frac{(-1)^{k+1} \cdot \xi_{k}^{(j)}}{q_{k}},$$ where $\{q_{k}\}$ is a sequence of positive integers,
$q_{k+1}\geq q_{k}(q_{k}+1),$  let $\{\xi_{k}^{(j)}\}$ be  sequences of random variables taking  values $0$ and $1$ with
probabilities $p_{0k}$ and $p_{1k}$ correspondingly.
Let $\{\xi^{(j)}\}$ be mutually independent random variables and $$\psi=\sum_{j=1}^{m}\xi^{(j)}.$$
Then the random variable $\psi$ has a singular distribution with
$$\dim_{H}\left(S_{\psi}\right)=0$$ and $$L_{\psi}=\limsup_{|t|\rightarrow\infty}|f_{\psi}(t)|>0.$$
\end{theorem}
\begin{proof}
If $\xi_{1}^{(j)}=i_{1}^{(j)},~\xi_{2}^{(j)}=i_{2}^{(j)}, \ldots, \xi_{n}^{(j)}=i_{n}^{(j)},~~
 i_{s}^{(j)}\in\{0,1\},~s=\overline{1,n},j=\overline{1,m},$ then
 $$\sum\limits_{s=1}^{n}\frac{(-1)^{s+1}\cdot i_{s}^{(1)}}{q_{s}}+\ldots+\sum\limits_{s=1}^{n}\frac{(-1)^{s+1}\cdot i_{s}^{(m)}}{q_{s}}-
 \left[\sum\limits_{s=n+1}^{\infty}\frac{1}{q_{s}}+\ldots+\sum\limits_{s=n+1}^{\infty}\frac{1}{q_{s}}\right]\leq $$ $$\leq\xi^{(1)}+\xi^{(2)}+\ldots+\xi^{(m)}
 \leq$$
 $$\leq\sum\limits_{s=1}^{n}\frac{(-1)^{s+1}\cdot i_{s}^{(1)}}{q_{s}}+\ldots+\sum\limits_{s=1}^{n}\frac{(-1)^{s+1}\cdot i_{s}^{(m)}}{q_{s}}+
 \left[\sum\limits_{s=n+1}^{\infty}\frac{1}{q_{s}}+\ldots+\sum\limits_{s=n+1}^{\infty}\frac{1}{q_{s}}\right]$$
 Taking into account that  $\sum\limits_{s=n+1}^{\infty}\frac{1}{q_{s}}<\frac{2}{q_{n+1}}$ (see property 6 of denominators of the Ostrogradsky series), we conclude that
 $$ \sum\limits_{s=1}^{n}\frac{(-1)^{s+1}\cdot i_{s}^{(1)}}{q_{s}}+\ldots+\sum\limits_{s=1}^{n}\frac{(-1)^{s+1}\cdot i_{s}^{(m)}}{q_{s}}- m \cdot \frac{2}{q_{n+1}} \leq \psi \leq$$ $$\leq \sum\limits_{s=1}^{n}\frac{(-1)^{s+1}\cdot i_{s}^{(1)}}{q_{s}}+\ldots+\sum\limits_{s=1}^{n}\frac{(-1)^{s+1}\cdot i_{s}^{(m)}}{q_{s}} + m \cdot \frac{2}{q_{n+1}}$$

The random variable  $\xi^{(1)}_k+\xi^{(2)}_k+\ldots+\xi^{(m)}_k$  takes values from the set $\{0,1,\ldots,m\}$. So,  for any $n\in N$ the spectrum $S_{\psi}$ of the random variable $\psi$ can be covered by  $(m+1)^{n}$ intervals of length $\frac{4m}{q_{n+1}}$.
The $\alpha-$ volume of this covering  is equal to
$$(m+1)^{n}\cdot\left(\frac{4m}{q_{n+1}}\right
)^{\alpha}=\left(4m\right)^{\alpha}\cdot\frac{\left(m+1\right)^{n}}{{q_{n+1}}^{\alpha}}
 \leq \left(4m\right)^{\alpha}\cdot\frac{\left(m+1\right)^{n}}{\left(2^{2^{n-2}}\right)^{\alpha}}\rightarrow 0 ~~ (n\rightarrow\infty),~ \forall m \in \N. $$
So,
$$H^{\alpha}\left(S_{\psi}\right)=0,~~\forall~\alpha>0.$$
Hence,
$$\dim_{H}\left(S_{\psi}\right)=\inf\left\{\alpha:~H^{\alpha}\left(S_{\psi}\right)=0\right\}=0.$$
Since the random variable $\xi^{(1)}, \xi^{(2)}, ..., \xi^{(m)} $ are mutually independent, from general properties of the Fourier-Stieltjes transform it follows that
 $$f_{\psi}(t)=\prod\limits_{j=1}^{m} f_{\xi}^{(j)}(t)=\left(f_{\xi}^{(j)}\right)^{m}=\left(f_{\xi}(t)\right)^{m}.$$
Applying results of theorem \ref{thm:ck}, we get the desired statement:  $$\limsup_{|t|\rightarrow 0}|f_{\psi}(t)|=\limsup_{|t|\rightarrow 0}|f_{\xi}(t)|^{m}=
\left(\limsup_{|t|\rightarrow 0}|f_{\xi}(t)|\right)^{m}>0.$$
\end{proof}

\begin{remark*}

1) From the proof of the latter theorem it follows that the random variable $\psi = \xi^{(1)}+...+\xi^{(k)}$ has an anomalously fractal singular  distribution even in the case where random variables  $\xi^{(1)}, ...,\xi^{(k)}$ are not independent.

2)  It is impossible to consider infinite autoconvolutions, because the resulting random series will diverge almost surely.

\end{remark*}

\subsection{General convolutions of  singular distributions generated by the second Ostrogradsky series}
Let $\{q^{(j)}_{k}\}$ be  sequences of positive integers,
$q^{(j)}_{k+1}\geq q^{(j)}_{k}(q^{(j)}_{k}+1),$  let $\{\xi_{k}^{(j)}\}$ be  sequences of random variables taking  values $0$ and $1$ with
probabilities $p^{(j)}_{0k}$ and $p^{(j)}_{1k}$ correspondingly, and let $$\xi^{(j)}=\sum_{k=1}^{\infty}\frac{(-1)^{k+1} \cdot \xi_{k}^{(j)}}{q^{(j)}_{k}}$$.

%Let $\{\xi^{(j)}\}$ be mutually independent random variables and
\begin{theorem}
%Let $\sum\limits_{k=1}  ^{\infty} $
If \begin{equation}\label{general convolutions}
\widetilde{\psi}=\sum_{j=1}^{m}\xi^{(j)},
\end{equation}
then the random variable $\widetilde{\psi}$ has an anomalously fractal  singular distribution.
\end{theorem}

\begin{proof}
If $\xi_{1}^{(j)}=i_{1}^{(j)},~\xi_{2}^{(j)}=i_{2}^{(j)}, \ldots, \xi_{n}^{(j)}=i_{n}^{(j)},~~
 i_{s}^{(j)}\in\{0,1\},~s=\overline{1,n},j=\overline{1,k},$ then
 $$\sum\limits_{s=1}^{n}\frac{(-1)^{s+1}\cdot i_{s}^{(1)}}{q^{(1)}_{s}}+\ldots+\sum\limits_{s=1}^{n}\frac{(-1)^{s+1}\cdot i_{s}^{(k)}}{q^{(m)}_{s}}-
 \left[\sum\limits_{s=n+1}^{\infty}\frac{1}{q^{(1)}_{s}}+\ldots+\sum\limits_{s=n+1}^{\infty}\frac{1}{q^{(m)}_{s}}\right]\leq $$ $$\leq\xi^{(1)}+\xi^{(2)}+\ldots+\xi^{(k)}
 \leq$$
 $$\leq\sum\limits_{s=1}^{n}\frac{(-1)^{s+1}\cdot i_{s}^{(1)}}{q^{(1)}_{s}}+\ldots+\sum\limits_{s=1}^{n}\frac{(-1)^{s+1}\cdot i_{s}^{(k)}}{q^{(m)}_{s}}+
 \left[\sum\limits_{s=n+1}^{\infty}\frac{1}{q^{(1)}_{s}}+\ldots+\sum\limits_{s=n+1}^{\infty}\frac{1}{q^{(m)}_{s}}\right].$$
 It is clear that $$\sum\limits_{s=n+1}^{\infty}\frac{1}{q^{(1)}_{s}}+\ldots+\sum\limits_{s=n+1}^{\infty}\frac{1}{q^{(m)}_{s}} \leq m \cdot \max\{\sum\limits_{s=n+1}^{\infty}\frac{1}{q^{(1)}_{s}}, ..., \sum\limits_{s=n+1}^{\infty}\frac{1}{q^{(m)}_{s}}\} \leq $$

 $$
 m \cdot \max\{\frac{2}{q^{(1)}_{n+1}}, ..., \frac{2}{q^{(m)}_{n+1}}\} \leq m \cdot \frac{2}{2^{2^{n-1}}}.
 $$

Since the random variable  $\frac{  \xi_{k}^{(1)}}{q^{(1)}_{k}} +  \ldots + \frac{\xi_{k}^{(m)}}{q^{(m)}_{k}}$  takes at most $2^m$ values, we conclude that the spectrum $S_{\widetilde{\psi}}$ of the random variable $\widetilde{\psi}$ can be covered by  $2^{mn}$ intervals of length $m \cdot \frac{4}{2^{2^{n-1}}}$.
The $\alpha-$ volume of this covering  is equal to
$$2^{mn}\cdot\left(\frac{4m}{2^{2^{n-1}}}\right)^{\alpha} \rightarrow 0 ~~ (n\rightarrow\infty),~ ~ \forall \alpha >0, ~\forall m \in \N.   $$
So, the Hausdorff measure $H^{\alpha}(S_{\widetilde{\psi}})$ of the spectrum of $\widetilde{\psi}$ is equal to zero for any positive $\alpha$. Therefore, $\dim_{H}\left(S_{\psi}\right)=0.$
\end{proof}

Let us now consider the case where
\begin{equation}\label{cond.infinite convolutions}
\sum\limits_{j=1}^{\infty} \frac{1}{q_k^{(j)}} < +\infty.
\end{equation}
In such a case the random variable
\begin{equation}\label{infinite convolutions}
\widetilde{\psi}_{\infty}=\sum_{j=1}^{\infty}\xi^{(j)} = \sum\limits_{j=1}^{\infty} \left(\sum\limits_{k=1}^{\infty} \frac{(-1)^{k+1} \xi_k^{(j)}}{q_1^{(j)}}\right)
\end{equation}
is correctly defined and it has a bounded spectrum:
 %Actually, the spectrum
  $S_{\widetilde{\psi}_{\infty}} \subset \left[- \sum\limits_{j=1}^{\infty} \frac{1}{q_1^{(j)}}, ~ \sum\limits_{j=1}^{\infty} \frac{1}{q_1^{(j)}}\right] $.

  \begin{theorem}
  The random variable $\widetilde{\psi}_{\infty}$ is of pure type, i.e., it is either purely discretely distributed or purely singularly continuously resp. purely absolutely continuously distributed.
  It has a pure discrete distribution if and only if
  \begin{equation}
  \prod\limits_{j=1}^{\infty} \left(\prod\limits_{k=1}^{\infty} \max_{k} \{p^{(j)}_{0k}, p^{(j)}_{1k}\}\right) >0.
  \end{equation}
\end{theorem}
\begin{proof}
From property 6 of denominators of the second Ostrogradsky series and condition \ref{cond.infinite convolutions} follows that  $\sum\limits_{j=1}^{\infty} \sum\limits_{k=1}^{\infty} \frac{1}{q_k^{(j)}}  \leq \sum\limits_{j=1}^{\infty} \frac{2}{q_1^{(j)}} < + \infty.$ So,  the random variable $\widetilde{\psi}_{\infty}$ can be represented in the following form:
\begin{equation}\label{reoreder infinite convolutions}
\widetilde{\psi}_{\infty} = \sum\limits_{d=1}^{\infty} \left(\sum\limits_{k+j=d}^{\infty} \frac{(-1)^{k+1} \xi_k^{(j)}}{q_k^{(j)}}\right),
\end{equation}
i.e., $\widetilde{\psi}_{\infty}$ can be represented as a sum of convergent series of independent discretely distributed random variables, and, therefore, from the Jessen-Wintner theorem it follows that $\widetilde{\psi}_{\infty}$ has the distribution of pure type.

The atom of maximal weight of the distribution of $\xi^{(j)}$ is equal to  $D_j :=\prod\limits_{k=1}^{\infty} \max\limits_{k} \{p^{(j)}_{0k}, p^{(j)}_{1k}\}$. So, from L\'evi theorem (\cite{Lev31}) it follows that $\widetilde{\psi}_{\infty}$ is purely discretely distributed if and only if
the product $\prod\limits_{j=1}^{\infty} D_j$ converges, which prove the theorem.
\end{proof}

\textbf{Remark.} If for some $j \in N $ the random variable $\xi^{(j)}$ has a continuous distribution, then $\widetilde{\psi}_{\infty}$ also has no atoms.
But it can happens that $\xi^{(j)}$ is pure atomic for all $j \in N$, but  $\widetilde{\psi}_{\infty}$ does not.

As has been shown before, the sum $\sum\limits_{j=1}^m \xi^{(j)}$ is always singular and the corresponding spectrum is of zero Hausdorff dimension.
In the limiting case the situation is more complicated. The random variable $\widetilde{\psi}_{\infty}$ can be absolutely continuous as well as singularly continuous with different values of the Hausdorff dimension of the spectrum.
Let us consider two "critical" cases.

\textbf{Example 1.} Let $q_1^{(j)}= \frac{1}{2^j}$, $p_{01}^{(j)}=\frac{1}{2}$.
In such a case   $$\widetilde{\psi}_{\infty} = \sum\limits_{j=1}^{\infty} \frac{\xi_1^{(j)}}{2^j} + \sum\limits_{j=1}^{\infty} \left(\sum\limits_{k=2}^{\infty} \frac{(-1)^{k+1} \xi_k^{(j)}}{q_k^{(j)}}\right).$$
Since the random variable  $\sum\limits_{j=1}^{\infty} \frac{\xi_1^{(j)}}{2^j}$ has uniform distribution on the unit interval, we get absolute continuity of the distribution of $\widetilde{\psi}_{\infty}$. In this case $\dim_H (S_{\widetilde{\psi}_{\infty}})=1.$

\textbf{Example 2.} Let $\{q_k\}$ be an arbitrary Ostrogradsky sequence of positive integers. Let $q_k^{(j)} = \frac{1}{q_{k+j}}$.
In such a case the random variable $\widetilde{\psi}_{\infty}$ has a singular distribution with anomalously fractal spectrum for any choice of probabilities $p^{(j)}_{ik}$.

% ----------------------------------------------------------------

\section{Fine fractal properties of the distribution of the random variable $\xi$}\label{sec:fractal.properties}

The main aim of this section is the study fine fractal properties of the distribution of the random variable $\xi$  generated by a given Ostrogradsky sequence (see (\ref{eq:rand.var.def}) to remind the exact definition). Let $\mu = \mu_{\xi}$ be the probability measure corresponding to $\xi$. The spectrum $S_{\mu}$ has zero Hausdorff dimension and, therefore, the Hausdorff dimension of the measure $dim_H(\mu) := \inf\limits \{ \dim_H(E), E
\in \mathcal{B}, \mu(E)=1 \}$  is also equal to zero. So, in such a case the classical Hausdorff dimension does not
reflect the difference between the spectrum and other essential supports (see, e.g., \cite{TorbinUMJ2005}) of the singular measure $\mu$.

So, study fine fractal properties of the distribution of $\xi$ it is necessary to apply more delicate tools than $\alpha-$dimensional Hausdorff measure and the corresponding Hausdorff dimension. To this end let us consider  the so called $h-$Hausdorff measures.
%%%, which are natural generalization of the  measure and dimension -
%%% It is sometimes desirable to have a sharper indication of dimension than just a
%%%number.
Let $h(t): R_+ \to R_+$ be a continuous increasing (non-decreasing) function such that $\lim\limits_{t\to 0} h(t) = 0.$  Usually the function $h$ is called the dimensional function or gauge function.
For a given set $E$, a given gauge function $h$ and a given $\varepsilon >0$, let
$$ H^{h}_{\varepsilon} (E) := \inf\limits_{|E_j| \leq \varepsilon} \sum\limits h(|E_j|), \bigcup_j E_j \supseteq E, $$
 where the infimum is taken over all $\varepsilon$-coverings $\{E_j\}$ of the set $E$.

 Since $ H^{h}_{\varepsilon_1} (E) \geq H^{h}_{\varepsilon_2} (E)$ for $\varepsilon_1 < \varepsilon_2$,  the following limit
  $$ H^{h}(E) := \lim\limits_{ \varepsilon \to 0} H^{h}_{\varepsilon} (E)$$ exists, and is said to be the h-Hausdorff measure (or $H^h-$measure) of the set $E$. For a given  set $E$ and a given gauge function $h$ the value $H^{h}(E)$ can be either zero or  positive and finite, or to be equal $+\infty$.
 If $h(t) = t^{\alpha}$, then we get the classical Hausdorff measure. If
$h_1$ and $h_2$ are dimension functions such that $\lim\limits_{t \to 0} \frac{h_1(t)}{h_2(t)} = 0$, then  $H^{h_1}(E ) = 0$ whenever $H^{h_2}(E ) < +\infty$ (see, e.g., \cite{Fal03, TuP} for details). Thus partitioning the dimension functions into those for which $H^h$ is finite and those
for which it is infinite gives a more precise information about fine fractal properties of a set $E$. An important example of this is Brownian motion in $R^3$ (see Chapter 16 of \cite{Fal03} for details). It can be shown that almost surely a Brownian path is of the Hausdorff dimension 2,  but their $H^2$-measure is equal to 0. More refined calculations show that such a path has positive and finite $H^h$-measure, where
$h(t) = t^2 \log (\log(\frac{1}{t}))$.

    So, our first aim of this section is to find a gauge function $h$ for the spectrum of the random variable $\xi$ and study fractal properties of the probability measure $\mu_{\xi}$ w.r.t. the measure $H^{h}$.  To this aim let us remind the notion of  the Hausdorff--Billingsley dimension.

Let $M$ be a fixed bounded subset of the real line. A family
$\Phi_M$ of intervals is said to be a \emph{fine covering family}
for $M$ if for any subset $E\subset M$, and for any
$\varepsilon>0$ there exists an at most countable
$\varepsilon$-covering $\left\{E_j\right\}$ of $E$, $E_j \in
\Phi_M$.
%%%i.e., $\forall E \subset M$, $\forall \varepsilon>0$ \
%%%$\exists \left\{E_j\right\}$ ($E_j\in \Phi_M$, $|E_j|\leq
%%%\varepsilon$): $E\subset \bigcup\limits_j E_j$.
A fine covering
family $\Phi_M$ is said to be \emph{fractal} if for the
determination of the Hausdorff dimension of any
subset $E \subset M$ it is enough to consider only coverings from
$\Phi_M$.

For a given bounded subset $M$ of the real line, let $\Phi_M$ be a
fine covering family for $M$, let $\alpha$ be a positive number
and let $\nu$ be a continuous probability measure. The
$\nu$-$\alpha$-Hausdorff--Billingsley measure of a subset $E
\subset M$ w.r.t. $\Phi_M$ is defined as follows:
\[
H^\alpha(E,\nu, \Phi_M) = \lim_{\varepsilon \to 0} \biggl\{
\inf_{\nu(E_j) \leq \varepsilon} \sum_j (\nu (E_j
))^\alpha\biggr\},
\]
where $E_j \in \Phi_M$ and $\bigcup\limits_j E_j \supset E$.

\begin{definition*}
The number $\dim_H (E, \nu, \Phi_M)= \inf \{ \alpha:
H^\alpha(E,\nu,\Phi_M)= 0\}$ is called the
\emph{Hausdorff--Billingsley dimension of the set $E$ with respect
to the measure $\nu$ and the family of coverings $\Phi_M$}.
\end{definition*}

\begin{remark*}
1) Let $\Phi_M$ be the family of all closed subintervals of the
minimal closed interval $[a,b]$ containing $M$. Then for any $E
\subset M$ the number $\dim_H(E, \nu,\Phi_M)$ coincides with the
classical Hausdorff--Billingsley dimension $\dim_H(E, \nu)$ of the
subset $E$ w.r.t. the measure $\nu$.

2) Let $M=[0,1]$, $\nu$ be the Lebesgue measure on $[0,1]$ and let
$\Phi_M$ be a fractal family of coverings. Then for any $E \subset
M$ the number $\dim_H (E, \nu, \Phi_M)$ coincides with the
classical Hausdorff dimension $\dim_H(E)$ of the
subset $E$.
\end{remark*}

Let $\Phi_M^\nu$ be the image of a fine covering family under the
distribution function of a probability measure $\nu$, i.e.,
$\Phi_M^\nu = \{ E': E' = F_\nu (E), E \in \Phi_M \}$. The following lemma has been proven in \cite{ABPT2}.
%gives necessary and sufficient condition for a fine covering family $\Phi_M$ to be has
\begin{lemma}\label{lem:fractal.cf}
A fine covering family $\Phi_M$ can be used for the equivalent
definition of the Hausdorff--Billingsley dimension of any subset
$E\subset M$ w.r.t. a measure $\nu$ if and only if the covering
family $\Phi_M^\nu$ can be used for the equivalent definition of
the classical Hausdorff dimension of any subset
$E'=F_\nu(E)$, $E\subset M$, i.e.,
\[
\dim_H (E, \nu, \Phi_M)=\dim_H (E, \nu) \quad \text{for any $E\subset
M$}
\]
if and only if the covering family $\Phi_M^\nu$ is fractal.
\end{lemma}

\begin{definition*}
The number
\[
\dim_H(\mu, \nu)= \inf\limits_{E \in B_{\eta}} \{ \dim_H(E, \nu), E \in \mathcal{B}, \mu(E)=1 \}
\]
is said to be the \emph{Hausdorff--Billingsley dimension of the
measure $\mu$ with respect to the measure $\nu$}.
\end{definition*}

To show the difference between the spectrum and
essential supports of the measure $\mu=\mu_{\xi}$ it is natural to  use the Hausdorff--Billingsley
dimension with respect to the measure $\nu^*$, where $\nu^*$ is
the probability measure, which is ``uniformly distributed'' on the spectrum of the measure $\mu_{\xi}$,
 i.e.,
$\nu^*$ is the probability measure
corresponding to the random variable
\begin{equation}\label{eq:rand.var xi^*.def}
\xi^*=\sum_{k=1}^\infty
\frac{(-1)^{k+1}\xi_k^*}{q_k},
\end{equation}
where $\xi_k^*$ are independent random variables taking the values $0$ and $1$ with probabilities
$p^{*}_{0k}$ and $p^{*}_{1k}$ such that for any $i \in \{0,1\}$:

$p^{*}_{ik}=0$ if and only if $p_{ik}=0$;

$p^{*}_{ik}=1$ if and only if $p_{ik}=1$;

$p^{*}_{ik}=\frac{1}{2}$ if and only if  $p_{ik} \in (0,1)$.

\textbf{Remark.} The measure $\nu^*$ can be considered as a  substitute of the Lebesgue measure on the set $S_{\mu}$.
If $p_{ik} \in (0,1), \forall i \in \{0,1\}, \forall k \in N $, then the measure $\nu^*$ is
 uniformly distributed not only on the spectrum $S_{\mu}$, but also  on the set $C_r$ of all incomplete sums of the second Ostrogradsky series $\sum\limits_{k=1}^{\infty} \frac{(-1)^{k+1}}{q_k}$. But, generally speaking, the spectrum  $S_{\mu}$ of the measure $\mu_{\xi}$ can be essentially smaller than the set $C_r$. It is clear that measures $\mu$ and $\nu^{*}$ have a common spectrum.

 \begin{theorem}\label{singularity of mu wrt nu}
 The measure $\mu$ is absolutely continuous w.r.t. the measure $\nu^{*}$  if and only if
 \begin{equation}
 \sum\limits_{k:~ p_{0k} \cdot p_{1k}>0 } (1- 2 p_{0k})^2 < +\infty;
 \end{equation}
 The measure $\mu$ is singularly continuous w.r.t. the measure $\nu^{*}$  if and only if $D=\prod\limits_{k=1}^{\infty} \max\{p_{0k}, p_{1k}\} =0$ and
 \begin{equation}
 \sum\limits_{k:~ p_{0k} \cdot p_{1k}>0 } (1- 2 p_{0k})^2 = +\infty.
 \end{equation}
 \end{theorem}
 \begin{proof}
 Let $\Omega_k = \{0,1\} $, $\mathcal{A}_k =
2^{\Omega_k}$. We define measures $\mu_k$ and $\nu_k$ in the
following way: $$   \mu_k (i) = p_{ik};  \nu_k (i)=
p^{*}_{ik}, ~~i \in \Omega_k.$$ Let $$ (\Omega, \mathcal{A}, \overline{\mu})=
\prod_{k=1}^\infty (\Omega_k , \mathcal{A}_k, \mu_k ),
(\Omega, \mathcal{A}, \nu)= \prod_{k=1}^\infty (\Omega_k,
\mathcal{A}_k, \nu_k)$$ be the infinite products of probability
spaces, and let us consider the measurable mapping $f: \Omega \to
R^1$ defined as follows: $$ \forall \omega = ( \omega_1,
\omega_2, ..., \omega_k, ...) \in \Omega,  f(\omega)= x=
 \sum\limits_{k=1}^{\infty} \frac{(-1)^{k+1} \omega_k}{q_k}.$$

We define the measures $\widetilde{\mu}$ and $\widetilde{\nu}$ as  image measures of
$\overline{\mu}$ resp. $\nu$ under $f$: $$ \widetilde{\mu} (B):= \overline{\mu} (f^{-1} (B)) ;
 \widetilde{\nu}(B) := \nu (f^{-1}(B)),  B \in \mathcal{B}.$$

It is easy to see that $\widetilde{\nu}$ coincides with the  measure
$\nu^{*}$, and $\widetilde{\mu}$ coincides with $\mu$. This mapping is bijective and bi-measurable.
Therefore (see, e.g., \cite{AT}), the measure $\mu$ is absolutely continuous (singular)
with respect  to the  measure $\nu^{*}$  if and only if the measure $\widetilde{\mu}$
is absolutely continuous (singular) with respect to the measure
$\widetilde{\nu}$. By construction, $\mu_k \ll \nu_k$,
$\forall k \in N$. By using Kakutani's theorem \cite{KAK}, we have
\begin{equation}
 \mu_\xi \ll \lambda \,\,\, \Leftrightarrow
\prod_{k=1}^\infty \left( \int_{\Omega_k} \sqrt{\frac{ d \mu_k}{d \nu_k}}
d \nu_k \right) >0 \,\,\, \Leftrightarrow \,\,\, \prod_{k=1}^\infty \left(
\sum_{i \in \Omega_k } \sqrt{p_{ik} q_{ik}}\right)>0,
\label{eqv}\end{equation}
\begin{equation} \label{sing}
  \mu_\xi \perp \lambda \,\,\, \Leftrightarrow \,\,\, \prod_{k=1}^\infty  \left(\int_{\Omega_k} \sqrt{\frac{ d \mu_k}{d \nu_k}}
d \nu_k \right)=0   \Leftrightarrow \,\,\, \prod_{k=1}^\infty
\left( \sum_{i \in \Omega_k } \sqrt{p_{ik} q_{ik}}\right) = 0.
\end{equation}

It is clear that $\sum\limits_{i \in \Omega_k } \sqrt{p_{ik} p^*_{ik}} = 1$ if $p_{0k} \cdot p_{1k} =0$, and $\sum\limits_{i \in \Omega_k }\sqrt{p_{ik} q_{ik}} = \sqrt{\frac12 p_{0k}} + \sqrt{\frac12 p_{1k}}$ if $p_{0k} \cdot p_{1k} >0$.
Since $p_{0k} + p_{1k} =1$, it is not hard to check that the product
 $
 \prod\limits_{k=1}^\infty
\left( \sum\limits_{i \in \Omega_k } \sqrt{p_{ik} q_{ik}}\right) = \prod\limits_{k:~p_{0k} \cdot p_{1k} >0 } \left( \sqrt{\frac12 p_{0k}} + \sqrt{\frac12 p_{1k}}
\right)
 $
converges to a positive constant if and only if the series  $\sum\limits_{k:~ p_{0k} \cdot p_{1k}>0 } (1- 2 p_{0k})^2$ converges, which proves the theorem.
 \end{proof}

\textbf{Remark.} Since the measures $\mu$ and $\nu^*$ have a common spectrum, the absolute continuity of $\mu$ w.r.t. $\nu^*$  means the equivalence of these measures (i.e., $\mu \ll \nu^*$ and $\nu^* \ll \mu $).

If $\max\{p_{0k}, p_{1k}\} =1$ for all large enough $k\in N$, then $\mu$ is discretely distributed with a finite number of atoms.
 So, in the sequel we shall assume that there are infinitely many indices $k$ such that $\max\{p_{0k}, p_{1k}\} <1$, which is equivalent to the continuity of the measure $\nu^{*}$.
\begin{theorem}\label{thm:dimHB.psi}
Let $h_n=-(p_{0n}\ln p_{0n}+p_{1n}\ln p_{1n})$ be the entropy of
the random variable $\xi_n$ and let
$H_n=h_1+h_2+\dots+h_n$. Then the Hausdorff--Billingsley dimension
of the measure $\mu_{\xi}$ with respect to the measure $\nu^*$ is equal to
\[
\dim_H(\mu, \nu^*)= \liminf_{n\to\infty} \frac{H_n}{g_n\ln2},
\]
where $g_n $ is the number of positive elements among $h_1, ..., h_n$.
% i.e., $g_n  = \# \{i: i \in N, i \leq n, p_{0k}\cdot p_{1k}>0\}$ .
\end{theorem}

\begin{proof}
Let  $r=\sum\limits_{k=1}^\infty
\frac{(-1)^{k+1}}{q_k}$, let $C_r$ be the set of all incomplete sums of this second Ostrogradsky series and  let $\widetilde{\Phi}_{C_r}$ be the
family of the above mentioned cylindrical intervals (see
Section~\ref{sec:random.incomplete.sum} for details), i.e.,
$\widetilde{\Phi}_{C_r} = \{ \Delta_{c_1
c_2 \dots c_k},  k\in \N, c_i \in \{0,1\} \},$ and $C_r =
\bigcap\limits_{k=1}^{\infty} \bigcup\limits_{c_i \in \{0,1\}}
\Delta_{c_1 c_2 \dots c_k}.$

Let $M=S_{\mu}\subset C_r$, and let $\Phi_{M}$ be the following subfamily of $\widetilde{\Phi}_{C_r}$:
$$ \Phi_{M} = \{ \Delta_{c_1
c_2 \dots c_n}, n\in \N ; ~~ c_j \in \{0,1\} ~\mbox{if}~ h_j >0, ~\mbox{and}~ c_j = i_{j} ~\mbox{if}~ h_j=0 ~\mbox{with}~ p_{i_{j}j}=1 \}.$$
It is clear that $ \Phi_{M}$ is a fine covering family for the spectrum $S_{\mu}$.

 Let $\Delta_{c_1 c_2 \dots c_n} \in \Phi_M$. Since $\nu^*(\Delta_{c_1 c_2 \dots
c_n})=2^{-g_n}$, the image $\Phi_M^{\nu^*} = F_{\nu^*}(\Phi_M)$
coincides with the fractal fine covering family consisting of
binary closed subintervals of $[0,1]$. So, from
Lemma~\ref{lem:fractal.cf} it follows that for the determination
of the Hausdorff--Billingsley dimension of an arbitrary subset of
$S_{\mu}$ w.r.t. $\nu^*$ it is enough to consider only
coverings consisting of cylindrical intervals from $ \Phi_{M}$.

Let $\Delta_n(x)= \Delta_{c_1(x) c_2(x)\dots
c_n(x) }$ be the cylindrical interval of the $n$-th rank containing
a point $x$ from the spectrum $S_{\mu}$. Then we have
\[
\mu (\Delta_n(x))= p_{c_1(x)1} \cdot p_{c_2(x)2} \cdot \ldots \cdot
p_{c_n(x)n}, \quad \nu^* (\Delta_n(x))= \frac{1}{2^{g_n}}.
\]

Let us consider the following expression
\[
\frac{\ln \mu
(\Delta_n(x))}{ \ln \nu^* (\Delta_n(x))}=
\frac{\sum\limits_{j=1}^n \ln p_{c_j(x)j}}{-g_{n} \ln2}.
\]

If $x= \Delta_{c_1(x) c_2(x)\dots c_n(x)\dots} \in S_{\mu}$ is chosen
stochastically such that $\Prob\{c_j(x)=i\}= p_{ij}>0$ (i.e., the
distribution of the random variable $x$ corresponds to the measure
$\mu$), then $\{ \eta_j \}= \{ \eta_j(x)\} = \{ \ln p_{c_j(x)j}\}$
is a sequence of independent random variables taking the values
$\ln p_{0j}$ and $\ln p_{1j}$ with probabilities $p_{0j}$ and
$p_{1j}$ respectively.
\[
\Expect(\eta_j) = p_{0j} \ln p_{0j}+p_{1j} \ln p_{1j}= -h_j, \quad
|h_j|\leq\ln2,
\]
\[
\Expect(\eta_j^2) = p_{0j}\ln^2 p_{0j} + p_{1j}\ln^2 p_{1j} \leq
c_0< \infty,
\]
and the constant $c_0$ does not depend on $j$.

If $p_{i_jj}=1$, then $\eta_j$ takes the value $0=E(\eta_j)$ with probability 1. So, in the sum $\eta_1(x)+\eta_2(x)+\dots +\eta_n(x)$ there are at most $g_n$ non-zero addends. It is clear that $g_n \to \infty$ as $n\to \infty$. By using the strong law of large numbers (Kolmogorov's theorem, see, e.g., ~\cite[Chapter III.3.2]{Shi96}), for
$\mu$-almost all points $x \in S_{\mu}$ the following equality holds:
\begin{equation*}
\lim_{n\to\infty} \frac{(\eta_1(x)+\eta_2(x)+\dots +\eta_n(x))-
\Expect(\eta_1(x)+\eta_2(x)+\dots + \eta_n(x))}{g_n}=0.
\end{equation*}
We remark that $\Expect(\eta_1+\eta_2+\dots +\eta_n)= -H_n$.

Let us consider the set
\begin{gather*}
A= \left\{ x: \lim_{n \to \infty} \left( \frac{\eta_1(x)+
\eta_2(x)+\dots + \eta_n(x)}{-g_n \ln2} - \frac{H_n}{g_n \ln2}
\right) =0 \right\}=\\
=\left\{ x: \lim_{n \to \infty} \frac{(\eta_1(x)+\eta_2(x)+\dots
+\eta_n(x))- \Expect(\eta_1(x)+\eta_2(x)+\dots + \eta_n(x))}{-g_n
\ln2} =0 \right\}.
\end{gather*}
Since $\mu(A)=1$, we have $\dim_H(A,\mu)=1$, and, therefore,
$\dim_H(A, \mu, \Phi_M)=1$.

Let us consider the following sets
\begin{align*}
A_1 &= \left\{ x: \liminf_{n \to \infty} \left( \frac{\eta_1(x)+
\eta_2(x)+\dots + \eta_n(x)}{-g_n \ln2} -
\frac{H_n}{g_n \ln2} \right) =0 \right\}; \\
A_2 &= \left\{ x: \liminf_{n \to \infty} \frac{\eta_1(x)+
\eta_2(x)+\dots + \eta_n(x)}{-g_n \ln2} \leq \liminf_{n \to
\infty} \frac{H_n}{g_n \ln2} \right\} =\\
 &= \left\{ x:
\liminf_{n \to \infty} \frac{ \ln \mu (\Delta_n(x))}{\ln \nu^*
(\Delta_n(x))} \leq \liminf_{n \to \infty} \frac{H_n}{g_n \ln2}
\right\};\\
A_3 &= \left\{ x: \liminf_{n \to \infty} \frac{\eta_1(x)+
\eta_2(x)+\dots + \eta_n(x)}{-g_n \ln2} \geq \liminf_{n \to
\infty} \frac{H_n}{g_n \ln2} \right\} =\\
&= \left\{ x: \liminf_{n \to \infty} \frac{ \ln \mu
(\Delta_n(x))}{\ln \nu^* (\Delta_n(x))} \geq \liminf_{n \to
\infty} \frac{H_n}{g_n \ln2} \right\}.
\end{align*}

It is obvious that $A \subset A_1$. We now prove the inclusions
$A_1 \subset A_3$ and $A \subset A_2$. To this end we use the
well-known inequality
\[
\liminf_{n \to \infty} (x_n - y_n)\leq \liminf_{n \to \infty} x_n
- \liminf_{n \to \infty} y_n
\]
holding for arbitrary sequences $\{x_n\}$ and $\{y_n\}$ of real
numbers.

If $x\in A_1$, then
\begin{gather*}
\liminf_{n \to \infty} \frac{\eta_1(x)+ \eta_2(x)+\dots +
\eta_n(x)}{-g_n \ln2} - \liminf_{n \to \infty} \frac{H_n}{g_n
\ln2} \geq \\
\geq \liminf_{n \to \infty} \left( \frac{\eta_1(x)+
\eta_2(x)+\dots + \eta_n(x)}{-g_n \ln2} - \frac{H_n}{g_n \ln2} \right)
=0.
\end{gather*}
So, $x\in A_3$.

If $x\in A$, then
\begin{gather*}
\liminf_{n \to \infty} \frac{H_n}{g_n \ln2} - \liminf_{n \to \infty}
\frac{\eta_1(x)+ \eta_2(x)+\dots + \eta_n(x)}{-g_n \ln2}
\geq \\
\geq \liminf_{n \to \infty} \left( \frac{H_n}{g_n \ln2} -
\frac{\eta_1(x)+ \eta_2(x)+\dots + \eta_n(x)}{-g_n \ln2} \right)= 0.
\end{gather*}
So, $x\in A_2$.

Let $D_0= \liminf\limits_{n \to \infty} \frac{H_n}{g_n \ln2}$. Since
$A\subset A_2 = \left\{ x: \liminf\limits_{n \to \infty} \frac{\ln
\mu (\Delta_n(x))}{\ln \nu^* (\Delta_n(x))} \leq D_0 \right\}$, we
have
\[
\dim_H (A,\nu^*, \Phi_M)\leq D_0.
\]
Since $A \subset A_3= \left\{ x: \liminf\limits_{n \to \infty}
\frac{ \ln \mu (\Delta_n(x))}{\ln \nu^* (\Delta_n(x))} \geq D_0
\right\}$, and $\dim_H(A,\mu,\Phi_M)=1$,  we deduce that
\[
\dim_H (A, \nu^*, \Phi_M)\geq D_0\cdot \dim_H (A,\mu, \Phi_M)= D_0
\] Therefore,
\[
\dim_H(A,\nu^*, \Phi_M)= \dim_H(A,\nu^*)= D_0.
\]

Let us now prove that the above constructed set $A$ is the
``smallest'' support of the measure $\mu$ in the sense of the
Hausdorff--Billingsley dimension w.r.t. $\nu^*$. Let $C$ be an
arbitrary support of the measure $\mu$. Then the
set $C_1= C \cap A$ is also a support of the same measure $\mu$,
and $C_1 \subset C$.

From $C_1 \subset A$, it follows that $\dim_H(C_1,\nu^*)\leq D_0$,
and
\[
C_1 \subset A \subset A_3= \left\{ x: \liminf_{n \to \infty}
\frac{ \ln \mu (\Delta_n(x))}{\ln \nu^* (\Delta_n(x))} \geq D_0
\right\}.
\]
Therefore,
\[
\dim_H(C_1, \nu^*) = \dim_H (C_1,\nu^*, \Phi_M) \geq D_0 \cdot
\dim_H(C_1,\mu, \Phi_M) = D_0 \cdot 1 = D_0.
\]
So, $\dim_H(C,\nu^*) \geq \dim_H (C_1,\nu^{*})= D_0
=\dim_H(A,\nu^*)$.
\end{proof}

\begin{corollary}
   From the construction of the measure $\nu^{*}$ it follows that $\dim_H (S_{\mu}, \nu^{*} )$ is always equal to 1.
\end{corollary}

\begin{corollary}
   The distribution function $h_1(t):=F_{\nu^*}(t) = \nu^*((-\infty, t)) = \nu^*([0, t))$ is the dimensional (gauge) function for the spectrum $S_\mu$, and the probability measure $\nu^*$ coincides with the restriction of the measure $H^{h_1}$ on the set $S_{\mu}$.
\end{corollary}

\begin{corollary}
   If the measure $\mu $ is absolutely continuous w.r.t. the measure $\nu^*$, then $\dim_H (\mu, \nu^{*} ) = 1$.
   \end{corollary}

\textbf{Remark.} From the latter theorem it follows that the Hausdorff-Billingsley dimension of the measure $\mu$ w.r.t. the measure $\nu^{*}$ can take any value from the set $[0,1]$.

\textbf{Examples.}

If $p_{0k}=\frac{1}{2k}, \forall k \in N$, then $\dim_H (\mu, \nu^{*} ) = 0;$

if $p_{0k}=\frac{1}{2} - \frac{1}{4\sqrt{k}}, \forall k \in N$, then $\dim_H (\mu, \nu^{*} ) = 1;$ but $\mu \bot \nu^*$.

if $p_{0,2k}=\frac{1}{2} - \frac{1}{4k}, p_{0,2k-1}=\frac{1}{2k},\forall k \in N$, then $\dim_H (\mu, \nu^{*} ) = \frac{1}{2}.$

\textbf{Remark.} To clarify how uniformly the measure $\mu$ is distributed on its spectrum, we shall  firstly check whether $\mu \ll \nu^*$. To  clarify how irregularly the measure $\mu$ is distributed on its spectrum, we should calculate the value of $\dim_H (\mu, \nu^{*} )$. The measure with smaller dimension can be considered as more irregularly distributed on its spectrum.

\bigskip
To stress the difference between the spectrum $S_{\mu}$ and the set $C_r$ of all incomplete sums of the second Ostrogradsky series, it is natural to use the Hausdorff-Billingsley dimension w.r.t. the measure $\nu_r$ which is uniformly distributed on the whole set $C_r$, i.e., $\nu_r$ is the probability distribution of the random variable
\[
\xi_r=\sum_{k=1}^\infty
\frac{(-1)^{k+1}\varepsilon_k}{q_k},
\]
where $\varepsilon_k$ are independent random variables taking the values $0$ and $1$ with probabilities
$\frac{1}{2}$ and $\frac{1}{2}$ correspondingly.

\begin{theorem}\label{theorem: dim wrt nu_r}
The Hausdorff--Billingsley dimension  of the measure $\mu_{\xi}$ w.r.t. $\nu_r$   is equal to
\begin{equation}\label{eq: dim mu wrt nu_r}
\dim_H(\mu, \nu_r)= \liminf_{n\to\infty} \frac{H_n}{n\ln2}.
\end{equation}
The Hausdorff--Billingsley dimension  of the spectrum  $S_\mu$ w.r.t. $\nu_r$  is equal to
\begin{equation}\label{eq: dim S_mu wrt nu_r}
\dim_H(S_\mu, \nu_r)= \liminf_{k\to\infty} \frac{N_k}{k},
\end{equation}
where $N_k= \# \{ j: j\leq k, p_{0j}p_{1j}>0\}$.
\end{theorem}

\begin{proof}
The proof of the first assertion of the theorem is completely similar to the the proof of theorem \ref{thm:dimHB.psi}.
To prove the second statement, let us consider an auxiliary probability measure $\nu^{*}$ which was defined above.

Generally speaking, the measures $\mu$ and $\nu^{*}$ do not coincide (moreover, they can be mutually singular), but their spectra are the same. The measure $\nu^{*}$ is uniformly distributed on
the spectrum of the initial measure. Therefore,
$$
\dim_H(\nu^*, \nu_r ) = \dim_H(S_{\nu^*}, \nu_r ) = \dim_H(S_{\mu}, \nu_r ).
$$
So, to determine the Hausdorff--Billingsley dimension of the
spectra of the initial measure w.r.t. the measure
$\nu_r$ it is enough to apply the first statement of  theorem \ref{theorem: dim wrt nu_r} to the
measure $\nu^*$.
%%%, taking into account that $\tilde{h}_k=0$, if
%%%$p_{0k}p_{1k} = 0$, and $\tilde{h}_k= \ln2$, if $p_{0k}p_{1k} \neq
%%%0$.
\end{proof}

\begin{corollary}
   The distribution function $h_2(t):= F_{\nu_r}(t)= \nu_r((-\infty, t)) = \nu_r([0, t))$ is the dimensional (gauge) function for the set $C_r$ of all incomplete sums.
   \end{corollary}

\textbf{Example.} Let $p_{0k}= \frac{1}{2}$ if $k =2^m,$ ~  and $p_{0k}= 0$ if $k \neq 2^m, ~ m \in N$.
In such a case:

$$ \dim_H({\mu, \nu^*}) = \dim_H ({S_{\mu}, \nu^*}) = 1,$$
and
$$ \dim_H({\mu, \nu_r}) = \dim_H({S_{\mu}, \nu_r}) = 0.$$

% ---------------------------------------------------------------

\section{Fractal dimension preservation}

In this section we develop third approach how to study level of "irregularity" of probability distributions whose spectra are of zero Hausdorff dimension.

Let $F$ be a probability distribution function and let $\gamma=\gamma_F$
be the corresponding probability measure with spectrum $S_F$. Let $\gamma^*$ be the probability measure which are
 uniformly distributed on $S_F$.
%%%
%%%\begin{definition*}
%%%a distribution function $F$ preserves the
%%%Hausdorff--Billingsley dimension (w.r.t. $\gamma^*$) on a set $A$
%%%if the Hausdorff--Billingsley dimension $\dim_H(E, \gamma )$ of any
%%%subset $E \subseteq A$ is equal to the Hausdorff--Billingsley
%%%dimension $\dim_H(E, \gamma^*)$.
%%%
%%%\end{definition*}

We say that a distribution function $F$ preserves the
Hausdorff--Billingsley dimension (w.r.t. $\gamma^*$) on a set $A$
if the Hausdorff--Billingsley dimension $\dim_H(E, \gamma )$ of any
subset $E \subseteq A$ is equal to the Hausdorff--Billingsley
dimension $\dim_H(E, \gamma^*)$.

If the probability distribution function $F$ strictly increases
(i.e., the spectrum $S_F$ is a closed interval), then
the above definition reduces to the usual definition of a
transformation preserving the Hausdorff--Besicovitch dimension
(see, e.g., \cite{APT04} for details).

Let $\xi$ and $\xi^*$ be random variables defined by
equalities~\eqref{eq:rand.var.def} and ~\eqref{eq:rand.var xi^*.def} respectively, and let
$\mu$ and $\nu^*$ be the corresponding probability measures.
It is easily seen that their spectra coincide. If, in
addition, $p_{ik}>0$ for any $i\in\set{0,1}$ and $k\in\N$, then
$S_{\nu^*}= S_\mu=C_r$ with $r=\sum\limits_{k=1}^\infty
\frac{(-1)^{k+1}}{q_k}$.

\textbf{Remark.} The distribution function $F_\mu$ does not preserve the
classical Hausdorff--Besicovitch dimension, because
$\dim_H(S_\mu)=0$ and $\dim_H(F_\mu (S_{\mu}))=1$. But, if the random variable $\xi$ is distributed "regularly" on its spectrum, then
$F_\mu$ can preserve the Hausdorff--Billingsley dimension on $S_{\mu}$.

\begin{theorem}
If there exists a positive constant $p_0$ such that $p_{ik}>p_0$,
then the distribution function $F_\xi$ of the random variable
$\xi$ preserves the Hausdorff--Billingsley dimension
$\dim_H(\cdot, \mu)$ on the spectrum $S_\mu$
if and only if the Hausdorff--Billingsley dimension of the measure $\mu$ with respect to the measure $\nu^*$ is equal to 1, i.e., if
\begin{equation}\label{eq:psi.DP}
\liminf_{n\to\infty} \frac{H_n}{g_n \ln2} = 1.
\end{equation}
\end{theorem}

\begin{proof}
\emph{Necessity.}  Let $\dim_H(\mu, \nu^*) =
\inf \{\dim_H(E, \nu^*), ~~E \in \mathcal{B}, \mu(E)=1\}$ be the
Hausdorff--Billingsley dimension of the measure $\mu$ w.r.t. the
measure $\nu^*$. From Theorem~\ref{thm:dimHB.psi} it follows
that $\dim_H(\mu, \nu^*) = \liminf\limits_{n\to\infty}
\frac{H_n}{g_n \ln2}$ with $H_n = h_1+h_2+\dots+h_n$ and $h_k =
-(p_{0k} \ln p_{0k} + p_{1k} \ln p_{1k})$. If $\dim_H(\mu, \nu^*)< 1$,
 then there exists a support $E$ of the measure $\mu$ such that
$\dim_H(\mu, \nu^*) \leq \dim_H(E, \nu^*)< 1$. Since
$\mu (E)=1$, we conclude that $\dim_H(E, \mu)=1 \not=
\dim_H(E, \nu^*)$, which contradicts the assumption of the
theorem.

\emph{Sufficiency.} It is easy to check n that $h_k \leq \ln 2$ for any
$k\in\N$ and the equality holds if and only if
$p_{0k}=p_{1k}=\frac{1}{2}$. Since $p_{ik}>0, \forall i \in \{0,1\}, \forall k \in N$, we have $g_n =n, \forall n \in N.$
Therefore,
condition~\eqref{eq:psi.DP} is equivalent to the existence of the
following limit
\begin{equation}\label{eq:HLlim}
\lim_{k \to \infty} \frac{h_1 +h_2 +\dots+h_k}{k \ln 2} = 1.
\end{equation}

For a given  $\varepsilon \in (0, \frac{1}{10}) $  let us consider the sets
\[
M_{\varepsilon,k}^+ = \left\{ j: j \in \N,\, j \leq k,\,
\abs{p_{0j}-\frac12} \leq \varepsilon \right\},
\]
and
\[
M_{\varepsilon,k}^- = \{ 1,2,\dots,k\} \setminus
M_{\varepsilon,k}^+= \left\{ j: j \in \N,\, j \leq k,\,
\abs{p_{0j}-\frac12} > \varepsilon \right\}.
\]

From $h_k\leq\ln2$ and condition~\eqref{eq:HLlim}, it follows that
\begin{equation*}
  \lim_{k \to \infty} \frac{|M_{\varepsilon,k}^+|}{k}=1, \qquad
  \lim_{k \to \infty} \frac{|M_{\varepsilon,k}^-|}{k}=0.
\end{equation*}

Let $\Delta_{\alpha_1 (x)\dots \alpha_k (x)}$ be the $k$-rank
cylindrical interval containing the point $x \in S_{\mu}$, and let us   consider the limit
\begin{equation*}
V(x) = \lim_{k \to \infty} \frac{\ln \mu (\Delta_{\alpha_1(x)
\dots \alpha_k(x) })}{\ln \nu^* (\Delta_{\alpha_1(x)\dots
\alpha_k (x)})}, ~ x \in
S_\mu
\end{equation*}
Firstly, we show that from $p_{ij} > p_0$ and from
condition~\eqref{eq:psi.DP}, it follows that the above limit
exists and $V(x)=1$ for any $x\in S_\mu$.
\begin{align*}
\frac{\ln \mu (\Delta_{\alpha_1(x) \dots \alpha_k(x) })}{\ln
\nu^* (\Delta_{\alpha_1(x)\dots \alpha_k (x)})}= \frac{ \ln
(p_{ \alpha_1(x)1}
%\cdot p_{\alpha_2(x) 2 }
 \cdot \dots \cdot
p_{\alpha_k(x) k }) }{\ln 2^{-k}}
= \frac{\sum\limits_{j \in M_{\varepsilon,k}^+} \ln p_{\alpha_j
(x) j} + \sum\limits_{j \in M_{\varepsilon,k}^-} \ln p_{\alpha_j
(x) j} }{-k \ln 2}.
\end{align*}

If $j \in M_{\varepsilon,k}^+$, then $\frac{1}{2}-\varepsilon \leq
p_{\alpha_j(x) j} \leq \frac{1}{2}+\varepsilon$, and,
consequently, there exists a number $C_{\varepsilon,k} \in \left[
\frac{1}{2}-\varepsilon, \frac{1}{2}+\varepsilon \right]$ such
that
\[
\sum_{j \in M_{\varepsilon,k}^+} \ln p_{\alpha_j (x) j} =
|M_{\varepsilon,k}^+| \cdot \ln C_{\varepsilon,k}.
\]

If $j \in M_{\varepsilon,k}^-$, then $p_0 \leq p_{\alpha_j(x) j}
\leq 1-p_0$. Therefore, there exists a number $d_{\varepsilon,k}
\in \left[p_0,1-p_0\right]$ such that
\[
\sum_{j \in
M_{\varepsilon,k}^-} \ln p_{\alpha_j (x) j} =
|M_{\varepsilon,k}^-| \cdot \ln d _{\varepsilon,k}.
\]

Let
\[
V^*(x):= \limsup\limits_{k \to \infty} \frac{\ln \mu
(\Delta_{\alpha_1(x)\dots \alpha_k(x)})}{\ln \nu^*
(\Delta_{\alpha_1(x)\dots\alpha_k(x)}) }= \limsup\limits_{k \to
\infty} \left( \frac{|M_{\varepsilon,k}^+|}{k} \cdot \frac{\ln
C_{\varepsilon,k}}{-\ln 2} + \frac{|M_{\varepsilon,k}^-|}{k} \cdot
\frac{\ln d_{\varepsilon,k}}{-\ln 2} \right).
\]
We have $V^*(x) \leq \frac{\ln \left( \frac{1}{2} -\varepsilon
\right)}{ - \ln 2}$, because $\frac{|M_{\varepsilon,k}^+|}{k} \to
1$, $\frac{1}{2}-\varepsilon \leq C_{\varepsilon,k} \leq
\frac{1}{2}+\varepsilon$, $\frac{|M_{\varepsilon,k}^-|}{k} \to 0$,
$p_0 \leq d_{\varepsilon,k} \leq 1-p_0$.

In the same way we can show that
\[
V_*(x) = \liminf_{k \to \infty} \frac{\ln \mu
(\Delta_{\alpha_1(x)\dots\alpha_k(x)})}{\ln \nu^*
(\Delta_{\alpha_1 (x)\dots \alpha_k(x)})} \geq \frac{\ln
\left(\frac{1}{2}+\varepsilon \right)}{-\ln2}, \quad \forall x \in
S_\nu, \quad \forall \varepsilon \in \left(0,\frac{1}{10}\right).
\]
Therefore $V(x)=1$ for any $x\in S_\mu$, and, by using
Billingsley's theorem~\cite{Bil61}, we have for all $E \subset
S_\mu$:
\[
 \dim_H (E, \nu^{*}, \Phi_M)= 1 \cdot \dim_H(E, \mu, \Phi_M),
\]
where $\Phi_M$ is a fine covering family  of cylindrical intervals $\Delta_{\alpha_1(x)\dots\alpha_k(x)}$ for the spectrum $S_{\mu}$.
The image of the family $\Phi_M$ under the distribution function $F_{\nu^{*}}$ coincides with the family of all binary subintervals of the unit interval. So, from the fractality of the family $F_{\nu^{*}}(\Phi_M)$ and from lemma \ref{lem:fractal.cf} it follows that  $ \dim_H (E, \nu^{*}, \Phi_M) = \dim_H (E, \nu^{*})$. The image of the family $\Phi_M$ under the distribution function $F_{\mu}$ coincides with the family of all $Q^{*}$-cylinders of the $Q^{*}$-expansion of real numbers with $q_{ik}= p_{ik}$ (see, e.g., \cite{AT2} for details). Since $p_{ik}>p_0>0$, from lemma 1 of the paper \cite{AT2}  and from our lemma \ref{lem:fractal.cf} it follows that $\dim_H(E, \mu, \Phi_M)=\dim_H(E, \mu)$. So, $\dim_H (E, \nu^{*}) = \dim_H(E, \mu), ~~ \forall ~E \subset
S_\mu$,  which proves the theorem.
\end{proof}

% ----------------------------------------------------------------

\section*{Acknowledgement}
This work was partially supported by DFG 436 UKR 113/97, EU project STREVCOMS, and by Alexander von Humboldt Foundation.
 % ----------------------------------------------------------------

\end{document}